\documentclass[10pt, reqno]{amsart}

\makeatletter
\@namedef{subjclassname@2020}{%
	\textup{2020} Mathematics Subject Classification}
\makeatother
\usepackage{hyperref}
\usepackage[T1]{fontenc}
\usepackage{amssymb}
\usepackage{amsmath}
\usepackage{amsthm}
\usepackage{enumerate}
\usepackage[hmargin=3.5cm,vmargin=3.5cm]{geometry}
\usepackage{caladea}

\newtheorem{theorem}{Theorem}[section]

\newtheorem{lemma}[theorem]{Lemma}

\newtheorem{corollary}[theorem]{Corollary}
\numberwithin{equation}{section}

\newcommand{\tupleh}{\textbf{h}}

\newcommand{\tuplex}{\textbf{x}}
\newcommand{\tupley}{\textbf{y}}

\newcommand{\tuplealpha}{\boldsymbol{\alpha}}

\newcommand{\tuplebeta}{\boldsymbol{\beta}}
\newcommand{\tuplegamma}{\boldsymbol{\gamma}}

\newcommand{\tuplelambda}{\boldsymbol{\lambda}}

\newcommand{\sinc}{\text{sinc}}
\newcommand{\Ncal}{\mathcal N_{s,\theta}^\tau(P)}

\title{Diophantine Inequalities of Fractional Degree}

\author{Constantinos Poulias}
\address{CP: School of Mathematics, Fry Building, Woodland Road, Clifton, Bristol BS8 1UG, UK.}\email{constantinos.poulias@gmail.com}

\begin{document}
	\date{}
	\maketitle
	%\makeatletter
	\begin{abstract}
		This paper is concerned with the study of diagonal Diophantine inequalities of fractional degree $ \theta ,$ where $ \theta >2$ is real and non-integral. For fixed non-zero real numbers $ \lambda_i $ not all of the same sign we write 
			\begin{equation*}
				\mathcal F (\tuplex)  = \lambda_1 x_1^\theta + \cdots + \lambda_s x_s^\theta.
			\end{equation*}
		For a fixed positive real number $ \tau $ we give an asymptotic formula for the number of positive integer solutions of the inequality $ |	\mathcal F (\tuplex) | < \tau $ inside a box of side length $P.$ Moreover, we investigate the problem of representing a large positive real number by a positive definite generalized polynomial of the above shape. A key result in our approach is an essentially optimal mean value estimate for exponential sums involving fractional powers of integers.
	\end{abstract}
	
\makeatother
%\setcounter{tocdepth}{1}  %% hides subsections from appearing in the table of contents %%
%\tableofcontents

\section{Introduction}
	
A central topic in analytic number theory with a long history and various applications is the study of solubility of Diophantine inequalities. In this paper we are concerned with diagonal Diophantine inequalities whose degree is a fractional power. Let us make this more precise. Suppose that $ \theta > 2$ is real and non-integral, and suppose that $ s $ is a positive integer. Let $ \lambda_1, \ldots , \lambda_s $ be fixed non-zero real numbers, not all of the same sign. Consider the generalised polynomial
	\begin{equation}  \label{eq2.1.1}
		\mathcal F (\tuplex) = \mathcal F(x_1, \ldots, x_s) =  \lambda_1 x_1^\theta + \cdots + \lambda_s x_s^\theta.
	\end{equation}
Suppose that $ \tau $ is a fixed positive real number. A first natural question one can pose is the following. Does the inequality 
	\begin{equation} \label{eq2.1.2}
		\left|  \mathcal F (\tuplex ) \right| < \tau
	\end{equation}
admit a solution $ \tuplex = (x_1, \ldots, x_s) $ in positive integers? Note that the assumption that not all of the coefficients $\lambda_i$ are of the same sign, is natural in order to study the solubility of inequality (\ref{eq2.1.2}), for otherwise one always has $ | \mathcal F (\tuplex) | \geq | \lambda_1 | + \cdots + | \lambda_s |,$ and thus it is clear that $ | \mathcal F (\tuplex) |$ fails to take arbitrarily small values. In the case where inequality (\ref{eq2.1.2}) admits infinitely many solutions in positive integers, one could additionally ask for the distribution of them. To formulate this, take $P$ to be an arbitrary large positive real number that eventually we let tend to infinity. With this parameter serving as a quantification measure for the size of solutions of (\ref{eq2.1.2}) we write $ \Ncal = \mathcal N_{s,\theta}^\tau (P;  \tuplelambda ) $ for the number of positive integer solutions $ \tuplex = ( x_1, \ldots, x_s) $ of (\ref{eq2.1.2}) with $ x_i \in [1, P] \hspace{0.05in} (1 \leq i \leq s).$ A standard heuristic argument suggests that one typically expects $ \gg  P^{s-\theta}$ such solutions to (\ref{eq2.1.2}). 

Throughout the paper we make use of standard notation in the field such as Vinogradov and Landau symbols. We recall this notation at the end of the introduction. For the sake of clarity let us mention here that for $ x \in \mathbb R $ we write $ \lfloor x \rfloor = \max \{ n \in \mathbb Z : n \leq x \}$ to denote the floor function. We may now proceed and state the first result of this paper, which reads as follows.
	
	\begin{theorem} \label{thm2.1.1}
		Suppose that $ \theta > 2 $ is  real and non-integral, and suppose further that $ s \geq \left( \lfloor 2\theta  \rfloor +1 \right) \left( \lfloor 2\theta  \rfloor +2 \right) + 1$ is a natural number. Then as $P \to \infty $ one has
			\begin{equation} \label{eq2.1.3}
				\Ncal = 2\tau \Omega(s, \theta ; \tuplelambda) P^{s-\theta} + o \left( P^{s-\theta}\right),
			\end{equation}	
		where 
			\begin{equation*}
				\Omega(s, \theta ; \tuplelambda) =  \left(\frac{1}{\theta}\right)^s | \lambda_1 \cdots \lambda_s|^{-1/\theta} C(s, \theta ; \tuplelambda) > 0
			\end{equation*}		
		with
			\begin{equation*}
				C(s, \theta ; \tuplelambda) = \int_{\mathcal U} \left( - \sigma_s( \sigma_1 \beta_1 + \cdots + \sigma_{s-1} \beta_{s-1} ) \right)^{1/ \theta -1} ( \beta_1 \cdots \beta_{s-1})^{1 / \theta -1} \normalfont {\text d} \tuplebeta, 
			\end{equation*}
		where $ \normalfont {\text d} \tuplebeta$ here stands for $ \normalfont {\text d} \beta_1 \cdots \normalfont {\text d} \beta_{s-1}$, and  $ \sigma_i = \lambda_i / | \lambda_i|,$ and $\mathcal U$ denotes the set of points of the box $ [0, | \lambda_1 | ] \times \cdots \times [0, | \lambda_{s-1} | ], $ satisfying the condition that 
			\begin{equation*} 
				- \sigma_s (  \sigma_1 \beta_1  + \cdots  + \sigma_{s-1} \beta_{s-1}) \in [0, | \lambda_s|].
			\end{equation*}
		 In particular, the inequality (\ref{eq2.1.2}) possesses a positive integer solution.
	\end{theorem}

As a first comment on the asymptotic formula (\ref{eq2.1.3}), let us remark that, as will be apparent to experts, the positivity of the real number  $  C(s, \theta ; \tuplelambda) $ follows immediately from the fact that the $ \sigma_i$ are not all of the same sign. In the special case where $ \theta \in \mathbb Q $ is a rational number greater than $2$ one can obtain a "special" family of solutions as follows. Let us write $\theta = p / q \in \mathbb Q $ for some $ p, q \in \mathbb N $ with $\text{gcd}(p,q)=1 .$ Take $ x_i = y_i^q \hspace{0.05in} (1\leq i \leq s)$ where $ y_i \in \mathbb N $ with $ y_i \asymp Y \asymp P^{1/q}.$ Then inequality (\ref{eq2.1.2}) takes the shape $ | \lambda_1 y_1^p + \cdots + \lambda_s y_s^p | < \tau.$ When the number of variables $s$ is large enough in terms of $p,$ as for example in \cite[Theorem 1]{freeman_lower_bounds}, one has that the number of solutions of this last inequality is $ \gg Y^{s-p}.$ Thus, the number of solutions of the inequality (\ref{eq2.1.2}) satisfies $ \mathcal N_{s,\theta}^\tau (P) \gg \left( P^{1/q} \right)^{s-p}.$  On the other hand, the number of solutions obtained in this way is $ o \left( P^{s- p/q} \right),$ and so the number of solutions obtained is very small compared with what is expected.
	
The first to consider studying additive problems with non-integral exponents is Segal in the early 1930's. In the papers \cite{segal_1933_I}, \cite{segal_1933_II} and  \cite{segal_1934_ineq}, Segal studied Waring's problem with non-integral exponents, and additionally (phrased slightly different in his work) considered the problem of solubility of the inequality 
	\begin{equation*}
		| x_1^\theta + \cdots + x_s^\theta - \nu | < \tau,
	\end{equation*}
with $ \theta > 2 $ real and non-integral and  $ 0 < \tau < \nu^{ - c(\theta)/\theta}, $ where $ 0 < c(\theta) < 1$ is a fixed number, depending only on $\theta.$ For large values of $\nu$ Segal showed the existence of a solution $\tuplex \in \mathbb N^s,$ provided that we are given $ s \geq s_0(\theta) $ variables, where $ s_0(\theta) \approx \theta(\lfloor \theta \rfloor+1) 2^{ \lfloor \theta \rfloor+1} +1. $ In Theorem \ref{thm2.1.3} below we improve this.
	
For questions and results on the interface between the fields of Diophantine inequalities and Diophantine approximation the interested reader can refer to the monograph \cite{baker_book_dioph_ineq}, which contains an exposition of some of the most pivotal results in that area, dating up to late 1980's. 

A great body of work in the existing literature is concerned with counting solutions inside a bounded box of side length $P$ to indefinite inequalities of the shape
	\begin{equation} \label{eq2.1.4}
		| \lambda_1 x_1^d + \cdots + \lambda_s x_s^d | < \tau,
	\end{equation}
where $  d \geq 2$ is a natural number, and at least one of the ratios $ \lambda_i / \lambda_j $ is irrational. This last irrationality assumption is necessary, for otherwise if all the coefficients are in rational ratio then one can clear out the denominators by multiplying  with the least common multiple which would reduce the inequality to an equation over the integers. The latter has been a separate area of research since the birth of the Hardy - Littlewood circle method in the early 1920's. 
The problem of the solubility of inequalities of the shape (\ref{eq2.1.4}) first appears in the literature with the seminal work of Davenport and Heilbronn \cite{davenport-heilbr-ineq} in 1946. In that paper the authors prove that any real indefinite diagonal quadratic form in $ s = 5 $ variables can take arbitrarily small values. Their method to prove that result, what now is called the Davenport--Heilbronn method, is a Fourier analytic method over the entire real line. It is important here to mention that the main theorem of \cite{davenport-heilbr-ineq} shows that there exist  arbitrarily large values of the parameter $P$ such that (\ref{eq2.1.4}) with $d=2$ and $ s=5$ is soluble with $ x_i \in[1,P] \cap \mathbb Z.$ 
More precisely, Davenport and Heilbronn prove their result for a sequence of arbitrarily large numbers $P,$ that depends essentially on the continued fraction expansion of the irrational ratio $\lambda_i/ \lambda_j.$ Thus, their conclusion would apply to boxes of side length $P$ whenever this parameter $P$ is a term of that specific sequence of values. This dependence was removed only in the early 2000's by Freeman.
Beginning with \cite{freeman_lower_bounds}, Freeman introduced a variant of the Davenport-Heilbronn method motivated by methods developed in \cite{bentkus_goetze_paper}. This allowed Freeman to show the existence of infinitely many non-trivial integer solutions in boxes of any sufficiently large side length $P,$ given roughly $ d \log d $ variables (when $d$ is large).
Later in \cite{freeman_asymp_formula}, Freeman established for the first time an asymptotic formula for the number of integer solutions of (\ref{eq2.1.4}) inside the box $[-P,P]^s,$ provided we have $ s \geq 2^d +1 $ variables. The results of \cite{freeman_lower_bounds} and \cite{freeman_asymp_formula} were refined by Wooley in \cite{wooley_dioph_ineq}. 
Since we are not dealing with an inequality of positive integral degree $d$ as in (\ref{eq2.1.4}), we finish here our rather short tour amongst results concerning that problem. The interested reader is directed to the papers of Freeman and Wooley for a general discussion. 
	
It is reasonable to expect that a conclusion as in Theorem \ref{thm2.1.1} would remain valid if instead of a homogenous inequality as of the type (\ref{eq2.1.2}) we count solutions to an inhomogeneous inequality of the shape
	\begin{equation} \label{eq2.1.5}
		| \mathcal F (\tuplex)  - L | < \tau,
	\end{equation}
with $ \mathcal F $ as in (\ref{eq2.1.1}) and $L$ being a given real number. We write $ \mathcal N_{s,\theta}^\tau (P ; \tuplelambda, L) $ to denote the number of positive integer solutions $ \tuplex $ of the inequality (\ref{eq2.1.5}) with $ x_i \in [1, P] \hspace{0.05in} (1 \leq i \leq s).$
Here, the generalised polynomial $\mathcal F $ could be either indefinite or definite. In the case where $ \mathcal F $ is indefinite there is no restriction on the size of $ L .$ However, one has to take boxes with side length $P$ being sufficiently large in terms of $s, \theta $ and the coefficients $ \lambda_i$ of $ \mathcal F.$
On the other hand, if $ \mathcal F $ is positive definite then one has to assume that $ L \asymp_{\tuplelambda} P^\theta.$ Namely, there exist suitable positive constants $ c(\tuplelambda), C(\tuplelambda)$ such that  $ L$ belongs to an interval of the shape $ c(\tuplelambda) P^\theta \leq L \leq C (\tuplelambda) P^\theta .$ 
As is to be expected, the counting function of such solutions satisfies the same kind of asymptotic formula as in Theorem \ref{thm2.1.1}. The minor adjustments of the proof are postponed until section \ref{section_inhomog_case}. 
	
	\begin{theorem} \label{thm2.1.2}
		Suppose that $ \mathcal F $ is indefinite and let $ L$ be a fixed real number. Suppose further that $ \theta > 2 $ is real and non-integral, and that $ s \geq \left( \lfloor 2\theta  \rfloor +1 \right) \left( \lfloor 2\theta  \rfloor +2 \right) + 1$ is a natural number. Then as $P \to \infty $ one has
			\begin{equation*}
				\mathcal N_{s,\theta}^\tau (P ; \tuplelambda, L) = 2 \tau \Omega(  s,\theta ;  \tuplelambda) P^{s- \theta} + o \left( P^{s- \theta} \right),
			\end{equation*}
		where $ \Omega(  s, \theta ;  \tuplelambda)  $ is a positive real number depending only on $s,\theta$ and the coefficients $\lambda_i.$
	\end{theorem}
	
One can refer to \cite{parsell_ineq_in_primes} for such a conclusion for linear forms over primes, and to \cite{freeman_inhomog_ineq} for a general result concerning additive inhomogeneous inequalities of integral degree $ d \geq 2.$
	
More interesting is the case where $ \mathcal F $ is positive definite. In such a case the problem is reformulated as a problem of representing arbitrarily large numbers by the generalised polynomial $ \mathcal F .$ Instead of counting solutions inside a box, we can count solutions that represent an arbitrary large real number. That is to say, for a positive real number $ \nu $ sufficiently large in terms of $s,\theta$ and the positive number $ \tau,$ we ask how many positive integer solutions are possessed by the inequality
	\begin{equation} \label{eq2.1.6}
		| \mathcal F(\tuplex) - \nu| < \tau.
	\end{equation}
We write $ \rho_s (\tau , \nu) = \rho_s (\tau , \nu ; \tuplelambda) $ to denote the number of positive integer solutions of (\ref{eq2.1.6}). One anticipates $\rho_s (\tau ,\nu) $ to be large when $ \tau $ is fixed and $ \nu$ is large. Our next result establishes an asymptotic formula for the counting function $ \rho_s (\tau, \nu).$
	
	\begin{theorem} \label{thm2.1.3}
		Suppose that $ \theta > 2 $ is real and non-integral, and that $ \tau \in (0,1]$ is a fixed real number. Suppose further that $ s \geq \left( \lfloor 2\theta  \rfloor +1 \right) \left( \lfloor 2\theta  \rfloor +2 \right) + 1$ is a natural number. Then as $ \nu \to \infty $ one has		
			\[
				\rho_s (\tau ,\nu) = 2 ( \lambda_1 \cdots \lambda_s )^{- 1 / \theta} \frac{ \Gamma \left( 1+ \frac{1}{\theta} \right)^s }{ \Gamma \left( \frac{s}{\theta} \right)} \tau \nu^{s/\theta -1} + o \left( \nu^{s/ \theta -1} \right).
			\]
	\end{theorem}
	
A word is in order regarding the conclusions of Theorems \ref{thm2.1.2} and \ref{thm2.1.3}. Though they look similar there is an essential difference between these two conclusions. As we already mentioned, in the situation of Theorem \ref{thm2.1.2}  we count solutions of an inequality inside a box, while in the situation covered by Theorem \ref{thm2.1.3} we aim to "represent" a large positive number by the generalised polynomial $ \mathcal F.$ 
This difference is reflected in the shape of the asymptotic formulae we establish. In the indefinite case we consider boxes of arbitrarily large side length $P,$ while in the definite case covered by Theorem \ref{thm2.1.3}, the main term in the asymptotic formula is limited by the size of the real number $\nu$ we wish to represent, since there is a natural height restriction imposed on a solution $ \tuplex.$ 
This last observation is straightforward. Suppose that $ \lambda_i $ are all positive and suppose that we aim to represent a large positive real number $\nu.$ Choose now $ P = 2 ( \lambda_1^{-1/\theta} + \cdots + \lambda_s^{-1/\theta} +1 ) \nu^{1/\theta}.$ Then for any solution $ \tuplex $ of  (\ref{eq2.1.6}) one has $ x_i \leq P \hspace{0.05in} (1\leq i \leq s).$ As a remark, we draw the attention of the reader to the recent works of Chow \cite{chow_cubes_shifts}, \cite{chow_waring_shifts} and Biggs \cite{biggs_waring_shifts} for the problem of representing a number by shifts of $d$th-powers where $ d \in \mathbb N.$ That is to say, for $ \tau$ a sufficiently large positive real number, they investigate the solubility of the inequality
	\begin{equation*}
		| ( x_1 - \mu_1)^d + \cdots + (x_s - \mu_s)^d - \tau | < \eta,
	\end{equation*}
in integers $ x_i > \mu_i ,$ where $ \mu_i $ are fixed real numbers with $ \mu_1$ being irrational and $ \eta $ being a positive real number.
	
From now on we focus on Theorem \ref{thm2.1.1}. It is possible even at this stage to illustrate the route we take to tackle the problem. For $ \alpha \in \mathbb R$ we define the exponential sum $f(\alpha) = f(\alpha ;P)$ via
	\begin{equation*}
		f(\alpha ; P) = \sum_{ 1 \leq x \leq P} e( \alpha x^\theta),
	\end{equation*}
where $ e(x)$ denotes $ e^{2 \pi i x}.$ As in Freeman's variant of the Davenport-Heilbronn method, we are seeking mean value estimates of the asymptotic shape $ \int_0^1 |f(\alpha)|^s \text d \alpha \ll P^{s-\theta}.$ The key mean value estimate that does the heavy lifting in the proof of Theorem \ref{thm2.1.1} is the following.   

	\begin{theorem} \label{thm2.1.4}
		Suppose that $ \theta > 2 $ is real and non-integral and that $\kappa \geq 1$ is a real number. Suppose further that $ t \geq \frac{1}{2} \left( \lfloor 2\theta  \rfloor +1 \right) \left( \lfloor 2\theta  \rfloor +2 \right) $ is a natural number. Then for any fixed $ \epsilon > 0$ one has
			\begin{equation*}
				\int_{-\kappa}^{\kappa} \left| f(  \alpha) \right|^{2t} \normalfont{ \text{d}} \alpha \ll \kappa  P^{ 2t - \theta + \epsilon}.
			\end{equation*}		
		The implicit constant on the above estimate depends on $ \epsilon, \theta$ and $t,$ but not on $ \kappa$ and $P.$ Furthermore, for $t > \frac{1}{2} \left( \lfloor 2\theta  \rfloor +1 \right) \left( \lfloor 2\theta  \rfloor +2 \right) $ one can take $ \epsilon  =0 .$
	\end{theorem}

The proof of Theorem \ref{thm2.1.4} proceeds by exploiting the Taylor expansion of the function $ x \mapsto x^\theta.$ The number $ \frac{1}{2} ( \lfloor 2 \theta \rfloor +1) ( \lfloor 2\theta \rfloor +2) $ stems from an application of the Main Conjecture in Vinogradov's mean value theorem to a system of degree $ k = \lfloor 2 \theta \rfloor +1.$  This idea seems to appear first in the work of Arkhipov and Zhitkov \cite{arkhipov_zitkov_warings_probl_non_integr_exp}. We follow the original approach of Arkhipov and Zhitkov. However, our treatment differs from that in \cite[Lemma 3]{arkhipov_zitkov_warings_probl_non_integr_exp} in two aspects. Firstly, we encounter from the very beginning an exponential sum with a smooth phase, while in \cite[Lemma 3]{arkhipov_zitkov_warings_probl_non_integr_exp} the authors deal with an exponential sum whose phase is the integer $\lfloor x^\theta \rfloor.$  Secondly, and most important, our treatment is a refinement of that presented in \cite[Lemma 3]{arkhipov_zitkov_warings_probl_non_integr_exp}. In the latter, the authors obtain an estimate which is $P^{1/2}$ away from the near optimal one. By contrast we establish an essentially optimal estimate in Theorem \ref{thm2.1.4}.
	
As a measure of comparison, note that when $ \theta = d$ is a natural number the latest developments in Vinogradov's mean value theorem by Wooley's Nested Efficient Congruencing method \cite[Corollary 14.7]{wooley_NEC} deliver the bound
	\begin{equation*}
		\int_0^1 | f(\alpha)|^s \text d \alpha \ll P^{s-d},
	\end{equation*}
provided $ s \geq s_0$ where 
	\begin{equation} \label{eq2.1.7}
		s_0 = d^2 -d  + 2 \lfloor \sqrt{  2d +2 } \rfloor - \theta(d),
	\end{equation}
with $\theta (d) $ defined via
	\begin{equation*}
		\theta (d) =
			\begin{cases}
				1, \hspace{0.1in} \text{when} \hspace{0.1in} 2d + 2 \geq \lfloor \sqrt{ 2d +2 }\rfloor^2 + \lfloor \sqrt{ 2d +2 }\rfloor, \\[10pt]
				2,  \hspace{0.1in} \text{when} \hspace{0.1in} 2d + 2 < \lfloor \sqrt{ 2d +2 }\rfloor^2 + \lfloor \sqrt{ 2d +2 }\rfloor. 
			\end{cases}
	\end{equation*}
Making use of the above mean value estimate combined with a Weyl type inequality as in \cite[Lemma 2.3]{wooley_dioph_ineq} one can show that $s_0 +1$ variables suffice to establish the anticipated asymptotic formula for the counting function $ \mathcal N_{s,d}^\tau (P ; \tuplelambda).$ Hitherto, in view of \cite[Theorem 11.3]{brued_kawada_wooley_8_annexe} one had to take $ s \geq 2d^2 $ when $d$ is large. Incorporating  (\ref{eq2.1.7}) into \cite{brued_kawada_wooley_8_annexe} reduces the number of variables needed to establish the asymptotic formula for $ \mathcal N_{s,d}^\tau (P ; \tuplelambda) $ by a factor of $2.$ We briefly mention here the following very interesting statistical result due to Br\"{u}dern and Dietmann. From a measure theoretic point of view, the anticipated asymptotic formula holds for almost all (admissible) real forms $ \lambda_1 x_1^d + \cdots + \lambda_s x_s^d,$ provided we have more than $2d$ variables. More precisely, in \cite{bruedern_dietmann_random_dioph_ineq} it is proven that given $ s > 2d $ variables then  for almost all (in the sense of Lebesgue measure) admissible values of the coefficients, there exists a positive real number $ C(s,d ; \tuplelambda )$ such that for all sufficiently large $P$ one has  
	\begin{equation*}
		\left| \mathcal N_{s,d}^\tau (P) - 2 \tau C(s,d ; \tuplelambda) P^{s-d} \right| <  P^{s-d - 8^{-2d}} ,
	\end{equation*}
uniformly in $ 0 < \tau \leq 1.$ It would be interesting to derive an analogue with the exponent $d$ replaced by an arbitrary positive fractional number $ \theta.$

Lastly, we encounter a weighted version of Theorem \ref{thm2.1.4}. For a sequence of complex numbers $( \mathfrak a_x)_{x \in \mathbb N} $ we write $ f_{\mathfrak a}(\alpha) = f_{\mathfrak a}(\alpha ; P)$ to denote the weighted exponential sum
	\begin{equation*}
		f_{\mathfrak a}(\alpha ; P) = \sum_{ 1 \leq  x \leq P} \mathfrak a_x e(\alpha x^\theta).
	\end{equation*}
Motivated by \cite{wooley_discr_fourier_restr_via_eff_congr} we seek for an inequality
	\begin{equation*}
		\|  f_{\mathfrak a} (\alpha; P)  \|_{L^{2s}} \leq C_{P}  \| \mathfrak a_x \|_{\ell^2},
	\end{equation*}
with the real number  $C_{P}$ depending at $P$ and being uniform in $ ( \mathfrak a_x)_{x \in \mathbb N}.$ Due to the fractional nature of $\theta,$ it is reasonable to expect a connection with Diophatine inequalities of the format (\ref{eq2.1.5}). To do so, one has to detect solutions of inequalities by means of an appropriate kernel function.  For $ \alpha \in \mathbb R $ we define the function
	\begin{equation} \label{eq2.1.8}
		\text{sinc} (\alpha) = 
			\begin{cases}
				\displaystyle \frac{\sin (\pi \alpha) }{ \pi \alpha}, & \hspace{0.1in} \text{when} \hspace{0.1in} \alpha \neq 0, \\[10pt]
				1, & \hspace{0.1in}  \text{when} \hspace{0.1in} \alpha=0,
			\end{cases}
	\end{equation}
and set $ K(\alpha) = \sinc^2 (\alpha)$ as in \cite{davenport-heilbr-ineq}. Our result reads as follows.
	
	\begin{theorem} \label{thm2.1.5}
		Suppose that $ \theta > 2 $ is real and non-integral, and suppose further that  $s \geq 2\left( \lfloor 2\theta  \rfloor+1 \right) \left( \lfloor 2\theta  \rfloor + 2 \right) + 2$ is a natural number. Then one has 	
			\begin{equation*}
				\int_{-\infty}^\infty \left|  f_{\mathfrak a}(\alpha) \right |^{2s} K(\alpha) \normalfont \text d \alpha \ll P^{s- \theta} \left( \sum_{ 1 \leq x \leq P} | \mathfrak a_x |^2 \right)^s.
			\end{equation*}		
	\end{theorem} 
	
In order to establish Theorem \ref{thm2.1.5} we apply an elementary argument and "double" the number of variables, aiming eventually to reduce to a Diophantine problem of representing a large positive real number by a generalised polynomial $ \mathcal F$ of the shape (\ref{eq2.1.1}). Thus, one would be able to make use of Theorem \ref{thm2.1.3}. This explains the fact that for the inequality recorded in Theorem  \ref{thm2.1.5}, we use twice as many number of variables needed in  Theorem \ref{thm2.1.3}. This is a "cheap" argument. With harder work one could possibly eliminate the factor 2 and half the number of variables needed. This requires more effort and is not the focus of this work. The trick of "doubling" the number of variables is a classical argument in harmonic analysis and goes back to at least Zygmund \cite{zygmund_1974}. More recently, it was used by Bourgain in the papers \cite{bourgain_fourier_transf_restr_I}, \cite{bourgain_fourier_transf_restr_II}, on discrete periodic Strichartz estimates.

\bigskip

\textit{Notation.} Below we collect a few pieces of notation that we use in the rest of the paper. For  $ x \in \mathbb R $ we write $ e(x) $ to denote $ e^{2 \pi i x}$ with $ i = \sqrt{-1}$ being the imaginary unit. For a complex number $z $ we write $ \overline z $ to denote its complex conjugate. For a function $f : \mathbb Z \to \mathbb C $ and for two real numbers $m, M,$ whenever we write
	\begin{equation*}
		\sum_{ m < x \leq M} f(x)
	\end{equation*}
the summation is to be understood over the integers that belong to the interval $(m, M].$ We make use of the standard symbols of Vinogradov and Landau. Namely, when for two functions $f,g$ there exists a positive real constant $ C $ such that $ | f(x) | \leq C | g(x) | $ for all sufficiently large $x$ we write $ f(x) = O (g(x))$ or $ f(x) \ll g (x).$ We write $ f \asymp g$ to denote the relation $ g \ll f \ll g.$  Furthermore, we write $ f(x) = o (g(x)) $ if $ f(x) / g(x) \to 0 $ as $ x \to \infty $ and we write $ f \sim g $ if $ f(x) / g(x) \to 1 $ as $ x \to \infty .$ Depending on the context the implicit constants in the Vinogradov and Landau symbols are allowed to depend on $\lambda_i, L, s, \theta, \tau$ and any $\epsilon > 0$ (which may change from one line to another), whenever such a quantity appears in our estimates. The implicit constants do not depend on $P.$ For a given real number $x$ we shall write $ \lfloor x \rfloor = \max \{ n \in \mathbb Z : n \leq x \}$ to denote the floor function. An expression of the shape $ m < \tuplex \leq M$ where $ m < M$ and $ \tuplex = (x_1, \ldots, x_n)$ is an $n$-tuple, is to be understood as $ m< x_1, \ldots, x_n \leq M.$

%%%%%% END OF INTRODUCTION %%%%%%%%%%%

%%%%%%%%%%%%% BEGINNING OF SET UP %%%%%%%%%%%%%%%%

\section{Set up } \label{section_set_up_overview}
	
We follow Freeman \cite{freeman_asymp_formula} in making use of appropriate kernel functions that allow one to bound the counting function $ \mathcal N_{s,\theta}^\tau (P)$ from above and below. We make use of the following technical lemma.
	
	\begin{lemma} \label{lemma2.2.1} 
	Fix a positive integer $h.$ Let $a$ and $b$ be real numbers with $0 <a<b.$ Then there is an even real function $K(\alpha) = K(\alpha; a,b)$ such that the function $ \psi$ defined by
		\[
			\psi( \xi) = \int_\mathbb R e( \xi \alpha) K(\alpha) \normalfont \text{d} \alpha
		\]
	satisfies 
		\begin{equation} \label{eq2.2.1}
			\psi(\xi) 			
				\begin{cases}
					\in[0,1] \hspace{0.1in} \text{for} \hspace{0.1in} \xi \in \mathbb R \\[10pt]				
					= 0 \hspace{0.1in} \text{for} \hspace{0.1in} | \xi | \geq b \\[10pt]				
					=1  \hspace{0.1in} \text{for} \hspace{0.1in} | \xi | \leq a.				
				\end{cases}
		\end{equation}	
	Moreover, $K$ satisfies the bound
		\begin{equation} \label{eq2.2.2}
			K(\alpha) \ll_h \min \left\{ b, |\alpha|^{-1}, | \alpha|^{-h-1}(b-a)^{-h} \right\}.
		\end{equation}	
	\end{lemma}
	
	\begin{proof}
		This is  \cite[Lemma 1]{freeman_asymp_formula}.
	\end{proof}

Set $ \widetilde \tau = \tau (\log P)^{-1}.$ We can now define the following two kernel functions
	\begin{equation} \label{eq2.2.3}
		K_{-} (\alpha) =  K(\alpha ; \tau - \widetilde \tau , \tau) \hspace{0.1in} \text{and} \hspace{0.1in} K_{+} (\alpha) = K (\alpha ; \tau , \tau + \widetilde \tau).
	\end{equation}
Note that by (\ref{eq2.2.2}) we have
	\begin{equation} \label{eq2.2.4}
		K_{\pm}  (\alpha) \ll_{\tau,h} \min \{1, | \alpha|^{-1}, \left( \log P \right)^h |\alpha|^{-h-1} \}.
	\end{equation}
	
The estimate (\ref{eq2.2.4}) is essential in the disposal of the set of trivial arcs. We make use of this for a particular choice of $h$ to be chosen at a later stage. We refer to $K_+, K_-$ as the upper and lower kernel respectively. The Fourier transform of $K_+$ provides us with an upper bound for $ \mathcal N_{s,\theta}^\tau(P)$ while the Fourier transform of $K_-$ provides a lower bound. To see this, let us write $\chi_\tau (\xi)$ for the indicator function of the interval $(-\tau, \tau),$ namely
	\begin{equation} \label{eq2.2.5}
		\chi_\tau (\xi) = 			
		\begin{cases}
			1, \hspace{0.1in} \text{when} \hspace{0.1in} | \xi | < \tau, \\				
			0, \hspace{0.1in} \text{when} \hspace{0.1in} | \xi| \geq \tau.
		\end{cases}	
	\end{equation}	
By (\ref{eq2.2.1}) one has that 
	\begin{equation*}
		\chi_{ \tau - \widetilde \tau } (\xi) \leq \int_{- \infty}^\infty e(\alpha \xi) K_{-} (\alpha) \text d \alpha \leq \chi_{ \tau} (\xi)
	\end{equation*}
	and
	\begin{equation*}
		\chi_{ \tau } (\xi) \leq \int_{- \infty}^\infty e(\alpha \xi) K_{+} (\alpha) \text d \alpha \leq \chi_{ \tau + \widetilde \tau} (\xi).	
	\end{equation*}
Consequently, one has
	\begin{equation}  \label{eq2.2.6}
		\int_{-\infty}^\infty e( \xi \alpha) K_{-}( \alpha) \text{d} \alpha \leq \chi_\tau (\xi) \leq \int_{-\infty}^\infty  e( \xi \alpha) K_{+}( \alpha) \text{d} \alpha.
	\end{equation}	
	
We take a moment to point out that the expression
	\begin{equation} \label{eq2.2.7}
		\left| \int_{-\infty}^\infty  e( \xi \alpha) K_{\pm}( \alpha) \text{d} \alpha - \chi_\tau (\xi) \right|
	\end{equation}
is zero when $ | | \xi| - \tau | > \widetilde \tau$ and at most $1$ for values of $ \xi$ such that $ | |\xi| - \tau | \leq \widetilde \tau .$ 
	
We are now equipped to explain how we sandwich the counting function $ \mathcal N_{s,\theta}^\tau(P).$ Recall that 
	\begin{equation*}
		f(\alpha) = \sum_{1 \leq x \leq P} e(\alpha x^\theta).
	\end{equation*}
We write $ f_i (\alpha) = f( \lambda_i \alpha)\hspace{0.05in} (1\leq i \leq s) $ and put
	\begin{equation} \label{eq2.2.8}
		R_{\pm} (P) = \int_{-\infty}^\infty f_1 ( \alpha) \cdots f_s(\alpha) K_{\pm} (\alpha) \text{d} \alpha.
	\end{equation}
Take now $ \xi = \lambda_1 x_1^\theta + \cdots + \lambda_s x_s^\theta $ in (\ref{eq2.2.1}). If we sum over $ 1 \leq \tuplex \leq  P$ and take into account (\ref{eq2.2.5}) and (\ref{eq2.2.6}), we obtain
	\begin{equation*}
		R_+ (P) \geq \sum_{ \substack{ 1 \leq  \tuplex \leq  P \\ | \mathcal F (\tuplex) | < \tau }} 1 = \Ncal,
	\end{equation*}
and
	\begin{equation*}
		R_- (P) \leq \sum_{ \substack{ 1 \leq \tuplex \leq  P \\ | \mathcal F (\tuplex) | < \tau }} 1 = \Ncal.
	\end{equation*}
Thus, we conclude that
	\begin{equation*}
		R_{-}(P) \leq  \Ncal \leq R_{+}(P).
	\end{equation*}
From the inequality above it is clear that in order to establish an asymptotic formula for $ \mathcal N_{s,\theta}^\tau (P)$ it suffices to obtain asymptotic formulae for the integrals $R_{\pm}(P)$ that are asymptotically equal. 

We now fix some notation. Put $ \gamma = \theta -\lfloor \theta \rfloor \in (0,1).$  We set
	\begin{equation} \label{eq2.2.9}
		\delta_0  = 4^{-\theta} \hspace{0.1in} \text{and} \hspace{0.1in} \omega = \min \left\{\frac{1-\gamma}{12}, 5^{-100 \theta} \right\}.
	\end{equation} 
We dissect the real line into three disjoint subsets as follows.
	\begin{itemize}
		\item[(i)] The major arc $ \mathfrak M $ around $0$ given by	
			\begin{equation*}
				\mathfrak M = \left\{ \alpha \in \mathbb R : | \alpha | < P^{-\theta + \delta_0} \right\}.	
			\end{equation*}
		\item[(ii)]  The minor arcs $ \mathfrak m $ given by	
			\begin{equation*}
				\mathfrak m = \left\{ \alpha \in \mathbb R : P^{-\theta + \delta_0}  \leq | \alpha | < P^{ \omega } \right\}.	
			\end{equation*}	
		\item[(iii)] The trivial arcs $ \mathfrak t $ given by
			\begin{equation*}
				\mathfrak t = \left\{ \alpha \in \mathbb R :  | \alpha | \geq  P^{ \omega } \right\}.
			\end{equation*}		
	\end{itemize}	
	
For a Lebesgue measurable set $ \mathcal B \subset \mathbb R $ we define
	\begin{equation} \label{eq2.2.10}
		R_\pm (P ; \mathcal B) = \int_{\mathcal B} f_1(\alpha) \cdots f_s (\alpha) K_\pm (\alpha) \text d \alpha.
	\end{equation}
So, by (\ref{eq2.2.8}) one has
	\begin{equation} \label{eq2.2.11}
		R_\pm (P) = R_\pm( P ; \mathfrak M) + R_\pm (P ; \mathfrak m) + R_\pm (P ; \mathfrak t).
	\end{equation}

%%%%%%%%%%% END OF SET-UP %%%%%%%%%%%%%%%%%

%%%%%%%%%% BEGINNING OF APPROXIMATELY TDI SYSTEMS %%%%%%

\section{An auxiliary mean value estimate} \label{section_key_estimate_proof}

In this section we prove Theorem \ref{thm2.1.4}. To do so, we first collect some auxiliary results that we need in our proof.

For $t,k \in \mathbb N$ we define the mean value
	\begin{equation*}
		J_{t,k} (P) = \int_{[0,1)^k} \left| \sum_{ 1 \leq x \leq P} e (\alpha_1 x + \cdots + \alpha_k x^k) \right|^{2t} \text d \tuplealpha.
	\end{equation*}
By orthogonality, one has that $J_{t,k} (P) $ counts the number of integer solutions of the system 
	\begin{equation*} 
	\sum_{i=1}^t ( x_i^j - x_{t+i}^j) =0 \hspace{0.1in} (1 \leq j \leq k),
	\end{equation*}	
with $ \tuplex \in [1, P ]^{2t}. $ The study of the mean value $J_{t,k}(P)$ goes back to the mid 1930's and Vinogradov \cite{vinogradovVMVT}. The central problem here is to find upper bounds for $J_{t,k}(P).$ The Main Conjecture in Vinogradov's mean value theorem, now a theorem after the work of Wooley \cite{wooley_VMVT_cubic} for $k=3,$ and Bourgain, Demeter and Guth \cite{bourgain_VMVT_proof}, for $ k \geq 4,$ reads as follows.
	
	\begin{theorem} \label{thm2.3.1}
		Suppose that $ t \geq \frac{1}{2}k (k+1)$ is a natural number. Then for any fixed $ \epsilon > 0 $ one has
			\begin{equation*}
				J_{t,k}(P) \ll P^{t+ \epsilon} + P^{2t - \frac{1}{2}k(k+1)}.
			\end{equation*}	
	\end{theorem}
	
	\begin{proof}
		See \cite[Corollary 1.3]{wooley_NEC}. An estimate weaker by a factor $ P^\epsilon$ can be found in \cite[Theorem 1.1]{wooley_VMVT_cubic} for $k=3$ and in \cite[Theorem 1.1]{bourgain_VMVT_proof} for $ k \geq 4.$
	\end{proof} 
	
In the proof of Theorem \ref{thm2.1.4}, we deal repeatedly with inequalities of the shape 
	\begin{equation} \label{eq2.3.1}
		| x_1^\theta + \cdots + x_t^\theta - x_{t+1}^\theta - \cdots  - x_{2t}^\theta | < \delta,
	\end{equation}
where $  \delta > 0$ is a fixed real number. Due to the fact that $ \theta $ is not an integer, one cannot count directly the solutions via the usual orthogonality relation over the interval $[0,1).$ As a surrogate, we make use of an auxiliary lemma which is a variant of \cite[Lemma 2.1]{watt-exp-sums-riemann-zeta-fnct2}. In order to state the lemma we first introduce some notation. Suppose that $I_1, I_2 \subset (0, \infty)$ are finite intervals, and suppose further that $ \mathcal S \subset \mathbb Z^2$ is a finite set of lattice points. We write $V_t(I_1, I_2 ; \delta)$ to denote the number of integer solutions of inequality (\ref{eq2.3.1}) with $ x_1, x_{t+1} \in I_1$ and $ x_i \in I_2$ for all $i \neq 1, t+1.$ Similarly, we write $V_t (\mathcal S, I_2 ; \delta)$ to denote the number of integer solutions of inequality (\ref{eq2.3.1}) with $ (x_1, x_{t+1}) \in \mathcal S$ and $ x_i \in I_2$ for all $i \neq 1, t+1.$ For $ \alpha \in \mathbb R $ and $i=1,2$ we put $H_i (\alpha) = H(\alpha ; I_i),$ where
	\begin{equation*}
		H (\alpha ; I_i) = \sum_{ x \in I_i} e(\alpha x^\theta).
	\end{equation*} 
Moreover, we write
	\begin{equation*}
		H_{\mathcal S} (\alpha) = \sum_{(x_1, x_{t+1}) \in \mathcal S} e\left(\alpha(x_1^\theta - x_{t+1}^\theta) \right).
	\end{equation*}
The lemma now reads as follows. We note here that if $ I_1 = I_2,$ then our result in $(ii)$ is a special case of \cite[Lemma 2.1]{watt-exp-sums-riemann-zeta-fnct2} with $ K=1$ and $ \omega = x^\theta $ in their notation.
	
	\begin{lemma} \label{lemma2.3.2}
		Define the number $ \Delta $ via the relation $ 2 \Delta \delta =1.$ 
		\begin{itemize}
			\item[(i)] One has
			\begin{equation*}
				V_t (\mathcal S, I_2; \delta) \ll\delta \int_{ - \Delta }^{\Delta} \left| H_{\mathcal S} (\alpha) H_2(\alpha)^{2t-2} \right| \normalfont{\text d} \alpha.
			\end{equation*}
			\item [(ii)] One has
			\begin{equation*}
				\delta \int_{ - \Delta }^{\Delta} \left| H_1 (\alpha)^2 H_2(\alpha)^{2t-2} \right| \normalfont{\text d} \alpha \ll V_t (I_1, I_2 ; \delta) \ll \delta \int_{ - \Delta }^{\Delta} \left| H_1 (\alpha)^2 H_2(\alpha)^{2t-2} \right| \normalfont{\text d} \alpha .
			\end{equation*}
		\end{itemize}
		The implicit constants in the above estimates are independent of $I_1, I_2, \mathcal S, \theta$ and $ \delta.$
	\end{lemma}
	
	\begin{proof}
	The argument proceeds as in \cite[Lemma 2.1]{watt-exp-sums-riemann-zeta-fnct2}. For $ x \in \mathbb R $ we define the functions
		\begin{equation*}
			K(\alpha) = \sinc^2 (\alpha) \hspace{0.3in} \text{and} \hspace{0.3in}	\Lambda (x) = \max \{ 0, 1 - | x| \},
		\end{equation*}	
	where recall from (\ref{eq2.1.8}) the definition of the $\sinc$ function. It is well known, one may see for example in \cite[Lemma 20.1]{davenport_book}, that for $ x, \xi \in \mathbb R $ one has			
		\begin{equation}  \label{eq2.3.2}
			K (\xi ) = \int_{- \infty}^\infty e(-x \xi) \Lambda (x) \text d x \hspace{0.1in}  \text{and} \hspace{0.1in} \Lambda (x) = \int_{- \infty}^\infty e (x \xi) K(\xi ) \text d \xi.
		\end{equation}	
	We make use of Jordan's inequality, which states that for $ 0 < x \leq \frac{\pi}{2}$ one has
		\begin{equation*}
			\frac{2}{ \pi} \leq \frac{\sin x}{x} < 1,
		\end{equation*}
		where the equality holds only if $ x = \pi / 2.$ For a proof of this inequality see \cite[p. 33]{mitrinovic_book}. Note here that for $ | \alpha| < \frac{1}{2}$ one has $K(\alpha) > 4/ \pi^2.$ 
		
	For ease of notation we set
		\begin{equation*}
			\sigma_{t,\theta} (\tuplex) = x_1^\theta + \cdots - x_{2t}^\theta \hspace{0.2in} \text{and} \hspace{0.2in} \xi = \frac{1}{2 \delta} \sigma_{t,\theta} (\tuplex).
		\end{equation*}	
		
	We first prove the upper bound in $(i).$ Let $ \tuplex$ be a tuple counted by $V_t \left( \mathcal S, I_2 ; \delta \right).$ By Jordan's inequality one has 
		\begin{equation*}
			\frac{\pi^2}{4} K( \xi) > 1.
		\end{equation*}
	Hence
		\begin{equation*}
			\begin{split}
				V_t \left( \mathcal S, I_2 ; \delta \right) \leq  \frac{ \pi ^2}{4} \sum_{\tuplex} K (\xi),
			\end{split}
		\end{equation*}
	where the summation is over tuples $ \tuplex$ with $ (x_1 ,x_{t+1}) \in \mathcal S$ and $ x_i \in I_2 \hspace{0.05in} (i \neq 1, t+1).$ Using now (\ref{eq2.3.2}) and making a change of variables by setting $ u = 2 \delta \alpha$ we successively obtain
		\begin{equation*}
			\begin{split}
				V_t \left( \mathcal S, I_2 ; \delta \right)  \leq \frac{ \pi^2}{4} \sum_{\tuplex} \int_{ -\infty}^{\infty} e (u \xi) \Lambda(- u) \text d u  = \frac{ \pi^2 \delta}{2} \sum_{\tuplex} \int_{ -\infty}^{\infty} e (\alpha \sigma_{t,\theta} (\tuplex)) \Lambda(- 2 \delta \alpha) \text d \alpha.
			\end{split}
		\end{equation*}
	One can interchange the order of integration with that of summation. This is valid since the integral is absolutely convergent and we have a finite sum. Note here that
		\begin{equation*}
			\sum_{\tuplex} e (\alpha \sigma_{t,\theta} (\tuplex)) = H_{\mathcal S} (\alpha) H_2(\alpha)^{2t-2}.
		\end{equation*}
	Moreover, for $ | \alpha| >  \frac{1}{2 \delta} $ one has  $ \Lambda( - 2 \delta \alpha) =0.$ Hence, by the triangle inequality we conclude that
		\begin{equation*}
			\begin{split}	
				V_t \left( \mathcal S, I_2 ; \delta \right)  & \leq \frac{ \pi^2 \delta}{2} \int_{ -\infty}^{\infty} \left| H_{\mathcal S} (\alpha) H_2(\alpha)^{2t-2} \right|  \Lambda(- 2 \delta \alpha) \text d \alpha \\[10pt]
				& \ll \delta \int_{ - \Delta}^{ \Delta} \left| H_{\mathcal S} (\alpha) H_2(\alpha)^{2t-2} \right| \text d \alpha.
			\end{split}
		\end{equation*}
		
	Next we prove $(ii).$ In order to establish the upper bound one may argue as in $(i),$ whereas now we make use of the product $H_1(\alpha)^2 H_2(\alpha)^{2t-2}.$ We give the proof of the lower bound. Let $ \tuplex$ be a tuple counted by $V_t \left(I_1, I_2 ; \delta \right).$ Then one has
		\begin{equation*}
			0 < \Lambda (2 \xi) < 1.
		\end{equation*}
	Thus, summing over $\tuplex$ with $ x_1,x_2 \in I_1$ and $ x_i \in I_2 \hspace{0.05in} (i \neq 1, t+1),$ and using (\ref{eq2.3.2}) we obtain
		\begin{equation*}	
			V_t \left( I_1, I_2 ; \delta \right)  \geq \sum_{\tuplex} \Lambda (2 \xi).
		\end{equation*}
	Invoking again (\ref{eq2.3.2}) and making a change of variables by setting $ u = \delta \alpha$ we successively obtain
		\begin{equation*} 
			\begin{split}
				V_t \left( I_1, I_2 ; \delta \right)  \geq \sum_{\tuplex } \int_{- \infty}^{\infty} e(2u \xi) K(u) \text d u 
				=  \delta \sum_{ \tuplex} \int_{- \infty}^{\infty} e(\alpha \sigma_{t, \theta} (\tuplex) ) K( \delta \alpha) \text d \alpha.
			\end{split}
		\end{equation*}
	Since we assume that $x_1, x_{t+1} \in I_1$ one has
		\begin{equation*}
			\sum_{\tuplex}  e(\alpha \sigma_{t, \theta} (\tuplex) )  = \left| H_1 (\alpha)^2 H_2(\alpha)^{2t-2} \right|.
		\end{equation*}	
	Changing the order of summation and integration the preceding inequality now delivers
		\begin{equation} \label{eq2.3.3}
			V_t \left( I_1, I_2 ; \delta \right)  \geq  \delta \int_{- \infty}^{\infty}  \left| H_1 (\alpha)^2  H_2(\alpha)^{2t -2} \right|  K( \delta \alpha) \text d \alpha.
		\end{equation}
	Next, using again Jordan's inequality and the positivity of the integrand we obtain
		\begin{equation*}
			\int_{- \infty}^{\infty} \left| H_1 (\alpha)^2  H_2(\alpha)^{2t -2} \right| K( \delta \alpha) \text d \alpha \geq  \frac{4}{ \pi^2} \int_{ - \Delta}^{\Delta}  \left| H_1 (\alpha)^2  H_2(\alpha)^{2t -2} \right| \text d \alpha.
		\end{equation*}
	Incorporating the above into (\ref{eq2.3.3}) yields 
		\begin{equation*}
			V_t \left( I_1, I_2 ; \delta \right) \gg \delta \int_{ - \Delta}^{\Delta}  \left| H_1 (\alpha)^2 H_2(\alpha)^{2t -2} \right| \text d \alpha,
		\end{equation*}
		which completes the proof.
	\end{proof}

From now on we set $ k = \lfloor 2 \theta \rfloor +1$ and for $ 1 \leq j \leq k $ we define the binomial coefficients
	\begin{equation*}
		b_j = \binom{\theta}{j} = \frac{ \theta (\theta-1) \cdots (\theta-j+1)}{ j \,!}.
	\end{equation*}
For a tuple $ \tupleh  = ( h_1, \ldots, h_k) \in \mathbb Z^k$ we write $ \mathcal H (\tupleh) = \mathcal H (h_1, \ldots, h_k) $ to denote the expression	
	\begin{equation} \label{eq2.3.4}
		\mathcal H (h_1, \ldots, h_k) = b_1 P^{\theta -1} h_1 + \cdots + b_k P^{\theta -k} h_k.
	\end{equation}

	\begin{lemma} \label{lemma2.3.3}
		 Suppose that $ \theta > 2$ is real and non-integral. Let  $ k = \lfloor 2 \theta \rfloor +1 ,$ and let $t$ be a given natural number. Suppose that $P \geq k^{2k}$ is a real number. We write $T(P)$ to denote the number of integer solutions of the inequality
			\begin{equation*}
				|\mathcal H (\tupleh) | \leq 2t
			\end{equation*}
		in the variables $ h_j,$ satisfying $ | h_j| \leq  t P^{j/2} \hspace{0.1in} (1\leq j \leq k).$ Then one has		
			\begin{equation*}
				T(P) \leq 4 (8t)^k  P^{ \frac{k(k+1)}{4} - \theta + \frac{1}{2}}.
			\end{equation*}
	\end{lemma}
	
	\begin{proof}
		This is  \cite[Lemma 1]{arkhipov_zitkov_warings_probl_non_integr_exp}. 
	\end{proof}
	
For technical reasons it is more convenient to work with exponential sums over dyadic intervals. For a positive real number $ X $ we write $ g(\alpha) = g(\alpha ; P)$ to denote the exponential sum
	\begin{equation} \label{eq2.3.5}
		g(\alpha ; P) = \sum_{ P < x \leq 2P} e(\alpha x^\theta).
	\end{equation}
We are now equipped to prove the key estimate of the paper. The following is a variant of Theorem \ref{thm2.1.4}, where now we are dealing with the exponential sum $g.$ One may recover Theorem \ref{thm2.1.4} by splitting the interval $[1,P]$ into $ O \left( \log P \right)$ dyadic intervals  and then apply the Theorem below.
	
	\begin{theorem} \label{thm2.3.4}
		Let $ \kappa \geq 1$ be a real number. Suppose that $ t \geq \frac{1}{2} ( \lfloor 2 \theta \rfloor +1) ( \lfloor 2\theta \rfloor +2) $ is a natural number. Then for any  fixed $ \epsilon > 0$ one has
			\begin{equation*}
				\int_{-\kappa} ^{ \kappa} \left| g( \alpha ) \right|^{2t} \normalfont{ \text{d}} \alpha \ll \kappa  P^{ 2t - \theta + \epsilon}.
			\end{equation*}
		The implicit constant on the above estimate depends on $ \epsilon, \theta$ and $t,$ but not on $ \kappa$ and $P.$ Furthermore, for $t > \frac{1}{2} \left( \lfloor 2\theta  \rfloor +1 \right) \left( \lfloor 2\theta  \rfloor +2 \right) $ one can take $ \epsilon  =0 .$
	\end{theorem}
	
	\begin{proof}
	 We set  $I = (P, 2P].$ Apply Lemma \ref{lemma2.3.2} with  $ I_1 = I_2 = I$ and $ \delta = \frac{1}{2 \kappa}.$ So one has		
		\begin{equation} \label{eq2.3.6}
			\frac{1}{2 \kappa} \int_{ - \kappa}^{ \kappa} \left| g( \alpha) \right|^{2t}  \text d \alpha \ll   V_t \left(I ; \frac{1}{2 \kappa} \right),
		\end{equation}
	where $  V_t \left(I ; \frac{1}{2 \kappa} \right)$ denotes the number of integer solutions of the inequality	
		\begin{equation*} \label{ineq_after_chang_variab}
			| x_1^\theta + \cdots + x_t^\theta - x_{t+1}^\theta - \cdots - x_{2t}^\theta | < \frac{1}{2 \kappa},
		\end{equation*}		
	with $ P < \tuplex \leq 2P.$ Since $ \kappa \geq 1 $ we plainly have that
		\begin{equation*}
			V_t \left(I ;  \frac{1}{2 \kappa} \right) \leq  V_t \left(I ; \frac{1}{2} \right),
		\end{equation*}
	where $ V_t \left( I ; \frac{1}{2} \right)$ denotes the number of integer solutions of the inequality
		\begin{equation*} 
		| x_1^\theta + \cdots + x_t^\theta - x_{t+1}^\theta - \cdots - x_{2t}^\theta | < \frac{1}{2},
		\end{equation*}	
 	with $ P < \tuplex \leq 2P.$ Hence by (\ref{eq2.3.6}) we obtain that
		\begin{equation} \label{eq2.3.7}
			\int_{- \kappa}^{\kappa} \left| g( \alpha) \right|^{2t}  \text d \alpha \ll  \kappa  V_t \left(I ;  \frac{1}{2} \right).
		\end{equation}
		
We define the interval		
	\begin{equation*}  
		\widetilde I = (  P,  P +  (\lfloor \sqrt P \rfloor + 1)  \sqrt P ].
	\end{equation*}
Note that $ I \subset \widetilde I.$ Moreover, for $ \alpha \in \mathbb R $ we write
		\begin{equation*}
			\widetilde g (a) = \sum_{ x \in \widetilde I} e(\alpha x^\theta).
		\end{equation*}
It is apparent that $  V_t \left(I ;\frac{1}{2} \right)$ is bounded above by the number of integer solutions of the inequality
		\begin{equation*}
			\left| \sum_{i=1}^t ( x_i^\theta - x_{t+i}^\theta) \right|< \frac{1}{2},	
		\end{equation*}
with $ x_1, x_{t+1} \in I$ and $ x_i, x_{t+i} \in \widetilde I \hspace{0.05in} (i \neq 1, t+1) .$ Denote this number by $ V_t (I, \widetilde I ; \frac{1}{2}).$
			
Apply now Lemma \ref{lemma2.3.2} with $ I_1= I$ and $I_2 = \widetilde I.$ This yields 
	\begin{equation} \label{eq2.3.8}
		\begin{split}
			V_t \left(I, \widetilde I ; \frac{1}{2} \right) \ll \int_{-1}^1 | g ( \alpha)|^2 | \widetilde g(\alpha)|^{2t - 2} \text d \alpha. 
		\end{split}
	\end{equation}
Putting together (\ref{eq2.3.8}) and the fact that $ V_t \left( I ; \frac{1}{2}  \right) \leq V_t \left( I, \widetilde I ; \frac{1}{2} \right),$ reveals that
	 \begin{equation} \label{eq2.3.9}
	 	V_t \left( I; \frac{1}{2} \right) \ll \int_{-1}^1 | g ( \alpha)|^2 | \widetilde g(\alpha)|^{2t - 2} \text d \alpha. 
	 \end{equation}
	
Our aim now is to bound the mean value on the right hand side of (\ref{eq2.3.9}). For a natural number $ \ell \geq 1$ we write 
		\begin{equation}  \label{eq2.3.10}
			P_\ell =  P + (\ell -1) \sqrt P,
		\end{equation}			
and set $ \widetilde I_\ell = ( P_\ell, P_{\ell+1}].$ Note that $ \widetilde I_\ell$ forms a cover of the interval $ \widetilde I$ consisting of subintervals of length $ \sqrt P.$  We record this in the following inclusion
		\begin{equation} \label{eq2.3.11}
			I \subset \widetilde I \subset \bigcup_{\ell =1}^{\lfloor \sqrt P \rfloor +1} \widetilde I_\ell.
		\end{equation}
For $ \alpha \in \mathbb R $ we now set
		\begin{equation*}
			\widetilde g_\ell (\alpha) = \sum_{ x \in \widetilde I_\ell } e(  \alpha x^\theta). 
		\end{equation*}
Incorporating the exponential sum $ \widetilde g_\ell  (\alpha),$ we deduce by the triangle inequality followed by an application of H{\"o}lder's inequality that
		\begin{equation*} 
			\begin{split}
				 \int_{-1}^1 | g ( \alpha)|^2 | \widetilde g(\alpha)|^{2t - 2} \text d \alpha & \leq \int_{-1}^1 | g( \alpha)|^{2}  \left(\sum_{ \ell =1}^{\lfloor \sqrt P \rfloor +1} \left|  \widetilde g_\ell (  \alpha)  \right| \right)^{2t-2} \text d \alpha \\[10pt]
				& \leq \left(\lfloor \sqrt P \rfloor +1\right)^{2t-3}  \sum_{ \ell =1}^{\lfloor \sqrt P \rfloor +1} \int_{-1}^1 \left| g(  \alpha) \right|^2 \left|  \widetilde g_\ell (\alpha)  \right|^{2t-2} \text d \alpha .
			\end{split}
		\end{equation*}
Invoking (\ref{eq2.3.9}), we infer that for some $ \ell_0$ with $ 1 \leq \ell_0 \leq \lfloor \sqrt P \rfloor +1 $ one has
		\begin{equation} \label{eq2.3.12}
			\begin{split}	
				V_t \left( I ; \frac{1}{2} \right)  & \ll \left(\lfloor \sqrt P \rfloor +1\right)^{2t-2}  \int_{-1}^1 \left| g(  \alpha) \right|^2 \left|  \widetilde g_{\ell_0} (\alpha)  \right|^{2t-2} \text d \alpha \\[10pt]
				& \ll P^{t-1} \int_{-1}^1 \left| g(  \alpha) \right|^2 \left|  \widetilde g_{\ell_0} (\alpha)  \right|^{2t-2} \text d \alpha.
			\end{split}
		\end{equation}

We now turn our attention to the mean value on the right hand side of (\ref{eq2.3.12}). One can apply Lemma \ref{lemma2.3.2} with $ I_1 = I$ and $ I_2 = \widetilde I_{\ell_0}.$ Then one has that
		\begin{equation} \label{eq2.3.13}
			\int_{-1}^1 \left| g(  \alpha) \right|^2 \left|  \widetilde g_{\ell_0} (\alpha)  \right|^{2t-2} \text d \alpha \ll V_t \left( I, \widetilde I_{\ell_0} ; \frac{1}{2} \right),
		\end{equation}
where $  V_t \left( I, \widetilde I_{\ell_0} ; \frac{1}{2} \right)$ denotes the number of integer solutions of the inequality
	\begin{equation} \label{eq2.3.14}
		\left| x_1^\theta - x_{t+1}^\theta + \sum_{i=2}^t (x_i^\theta - x_{t+i}^\theta) \right| < \frac{1}{2},
	\end{equation}
with $ 	x_1, x_{t+1} \in I$ and $x_i \in \widetilde I_{\ell_0} \hspace{0.05in} (i \neq 1, t+1).$

Recall  that $ \widetilde I_{\ell_0} = (P_{\ell_0} , P_{\ell_0+1}],$ where $ P_{\ell_0} =  P + (\ell_0 -1) \sqrt P .$ For each index $i \neq 1, t+1 $ we set
		\begin{equation*}
			y_i = x_i -  P_{\ell_0}. 
		\end{equation*}
Clearly one has $ 0 < y_i \leq \sqrt P .$ Upon noting that $ P \gg P_{\ell_0} \gg \sqrt P,$ an application of the mean value theorem of differential calculus yields for each index $ i \neq 1, t+1,$ that
	\begin{equation*}
		\begin{split}	
			| x_i ^\theta - x_{t+i}^\theta | = \left| (y_i +P_{\ell_0})^\theta - (y_{t+i} +P_{\ell_0})^\theta \right| \asymp P_{\ell_0}^{\theta -1}  | y_i - y_{t+i} | \ll  P^{\theta -1/2}.
		\end{split}
	\end{equation*}
By the triangle inequality, the above estimate leads to  
		\begin{equation*}	
			\left| \sum_{ \substack{i=2 \\ x_i \in \widetilde I_\ell}}^{t} (x_i^\theta - x_{t+i}^\theta) \right|  \ll P^{\theta -1/2}.
		\end{equation*}	
Invoking (\ref{eq2.3.14}) we now have that $ | x_1^\theta - x_{t+1}^\theta | \ll  P^{\theta -1/2}.$ On the other hand, an application of the mean value theorem of differential calculus yields $| x_1^\theta - x_{t+1}^\theta | \asymp | x_1 - x_{t+1} | P^{\theta -1}.$ Thus, we can conclude that $ | x_1 - x_{t+1} | \ll \sqrt P .$ One can rewrite this asymptotic estimate in the shape $ | x_1 - x_{t+1} | \leq C_1 \sqrt P ,$ where $ C_1 >0$ is a real number that depends at most on $t$ and $ \theta.$ In view of this new constraint one can return to inequality (\ref{eq2.3.14}) and count solutions subject to the constraints
		\begin{equation} \label{eq2.3.15}
			x_1, x_{t+1} \in I, \hspace{0.1in} | x_1 - x_{t+1} | \leq C_1  \sqrt P, \hspace{0.1in}  \text{and} \hspace{0.1in} x_i \in \widetilde I_{\ell_0} \hspace{0.05in} (i \neq 1, t+1).
		\end{equation}

The points $ x_1, x_{t+1} $ belong to the interval $I.$ Recalling the inclusion  (\ref{eq2.3.11}) we have that there are indices $ \ell_1$ and $ \ell_2$ for which 
	\begin{equation*}
		P_{\ell_1} < x_1 \leq P_{\ell_1 +1}  \hspace{0.1in} \text{and} \hspace{0.1in} P_{\ell_2} < x_{t+1} \leq P_{\ell_2 +1} .
	\end{equation*}		
Then, combining (\ref{eq2.3.15}) with  the definition (\ref{eq2.3.10}) of $P_\ell$ and using the fact that for each index $ \ell$ we have $ P_{\ell +1} - P_{\ell} = \sqrt P,$ one can deduce that
	\begin{equation*}
		\begin{split}
			C_1 \sqrt P \geq | x_1 - x_{t+1} | \geq \left| P_{\ell_1} - P_{\ell_2}  \right| - \sqrt P \geq \left( | \ell_1 - \ell_2| -1 \right) \sqrt P.
		\end{split}
	\end{equation*}
From the above computation we obtain that $ | \ell_1 - \ell_2 | \leq C_1 + 1.$

We now bound from above the number of integer solutions of the inequality (\ref{eq2.3.14}), under the constraint (\ref{eq2.3.15}) and the additional constraint we extracted just above. To do so, we make use of appropriate generating functions. We write $ \mathcal S \subset I \times I $ for the set of lattice points $ x_1, x_{t+1} \in I$ which satisfy $ |x_1 - x_{t+1}| \leq C_1 \sqrt P.$ By Lemma \ref{lemma2.3.2} with $\mathcal S$ as above and $I_2 = \widetilde I_{\ell_0}$ we deduce that
	\begin{equation} \label{eq2.3.16}
			V_t \left( I, \widetilde I_{\ell_0}; \frac{1}{2} \right) \ll \int_{-1}^1 \left| H_{\mathcal S} (\alpha) \widetilde g_{\ell_0} (\alpha)^{2t-2}  \right|,
	\end{equation} 
where 
	\begin{equation*}
		H_{\mathcal S} (\alpha) = \sum_{(x_1, x_{t+1}) \in \mathcal S} e \left( \alpha(x_1^\theta - x_{t+1}^\theta) \right).
	\end{equation*}
Using the cover $ \left( \widetilde I_\ell \right)_\ell$ and taking into account our previous conclusion that $ | \ell_1 - \ell_2| \leq C_1 + 1, $ we infer that
	\begin{equation*}
		\left| H_{\mathcal S}(\alpha) \right| \ll \sum_{\ell_1=1}^{ \lfloor \sqrt P \rfloor +1}  \sum_{ \substack{ \ell_2 = 1 \\ | \ell_1 - \ell_2| \leq C_1 + 1} }^{ \lfloor \sqrt P \rfloor + 1} | \widetilde g_{\ell_1} (\alpha) | | \widetilde g_{\ell_2} (\alpha)|.
	\end{equation*} 
Hence, for some $ 1 \leq \ell_1, \ell_2 \leq \lfloor \sqrt P \rfloor + 1$ one has
	\begin{equation*}
			\left| H_{\mathcal S}(\alpha) \right| \ll P^{ \frac{1}{2} } | \widetilde g_{\ell_1} (\alpha) | | \widetilde g_{\ell_2} (\alpha)|.
	\end{equation*}	

One can now bound above the right hand side of (\ref{eq2.3.16}). So we infer that
	\begin{equation} \label{eq2.3.17}
		V_t \left( I, \widetilde I_{\ell_0} ; \frac{1}{2} \right) \ll P^{ \frac{1}{2}} \int_{-1}^1 | \widetilde g_{\ell_1} (\alpha)  \widetilde g_{\ell_2} (\alpha) \widetilde g_{\ell_0} (\alpha)^{2t- 2}| \text d \alpha.
	\end{equation}	
Invoking the elementary inequality $ |z_1 \cdots z_n| \ll |z_1|^n + \cdots + | z_n|^n ,$ which is valid for all complex numbers, one has that
	\begin{equation*}
		 | \widetilde g_{\ell_1} (\alpha) \widetilde g_{\ell_2} (\alpha)  \widetilde g_{\ell_0} (\alpha)^{2t- 2}| \ll | \widetilde g_{\ell_1} (\alpha) |^{2t} +  | \widetilde g_{\ell_2} (\alpha) |^{2t} +  | \widetilde g_{\ell_0} (\alpha) |^{2t}.
	\end{equation*}  
Hence, (\ref{eq2.3.17}) delivers the estimate
	\begin{equation*}
			V_t \left( I, \widetilde I_{\ell_0} ; \frac{1}{2} \right) \ll P^{ \frac{1}{2}} \int_{-1}^1 | \widetilde g_\ell (\alpha)|^{2t} \text d \alpha,
	\end{equation*}
where $ \ell$ is one of the indices $ \ell_1, \ell_2, \ell_0.$ Incorporating this estimate into (\ref{eq2.3.13}) and recalling (\ref{eq2.3.12}), we deduce that
	\begin{equation}  \label{eq2.3.18}
		V_t \left( I ; \frac{1}{2} \right) \ll  P^{ t - \frac{1}{2}} \int_{-1}^1 | \widetilde g_\ell (\alpha)|^{2t} \text d \alpha. 
	\end{equation}
We emphasize here that our choice of $ 1 \leq \ell \leq \lfloor \sqrt P \rfloor + 1$ is now fixed.
		
In view of (\ref{eq2.3.7}) our aim in the rest of the proof is to bound the mean value appearing on the right hand side of (\ref{eq2.3.18}). Appealing to Lemma \ref{lemma2.3.2} with $ I_1 = I_2 = \widetilde I_{\ell} $ and $ \delta = \frac{1}{2}$ one has  
	\begin{equation} \label{eq.2.3.19}
		 \int_{-1}^1 | \widetilde g_\ell (\alpha)|^{2t} \text d \alpha \ll V_t \left( \widetilde I_\ell ; \frac{1}{2} \right),
	\end{equation}
where $  V_t \left( \widetilde I_\ell ; \frac{1}{2} \right) $ denotes the number of integer solutions of the inequality 
	\begin{equation} \label{eq2.3.20}
		\left| x_1^\theta + \cdots + x_t^\theta - x_{t+1}^\theta - \cdots - x_{2t}^\theta \right| < \frac{1}{2},  
	\end{equation}
with $ x_i \in \widetilde I_\ell.$ From now on we essentially follow \cite[Lemma 3]{arkhipov_zitkov_warings_probl_non_integr_exp}. We set $ Q_\ell = \lfloor P_\ell \rfloor $ and define $ y_i = x_i - Q_\ell \hspace{0.05in} (1\leq i \leq 2t).$ Note that 
		\begin{equation*} 
			0 < y_i < \lfloor \sqrt P \rfloor + 1 < Q_\ell.
		\end{equation*}	
This observation is immediate since by the definitions of $P_\ell$ and $Q_\ell$ one has
		\begin{equation} \label{eq2.3.21}
			0 \leq P_\ell - \lfloor P_\ell \rfloor < y_i \leq P_{\ell +1} - \lfloor P_\ell \rfloor = P_\ell - \lfloor P_\ell \rfloor + \sqrt P \leq \sqrt P + 1 < Q_\ell.
		\end{equation}
Then inequality (\ref{eq2.3.20}) takes the shape
		\begin{equation*} 
			\left| \left( y_1 + Q_\ell \right)^\theta + \cdots + \left( y_t + Q_\ell \right)^\theta - \left( y_{t+1} + Q_\ell \right)^\theta - \cdots  \left( y_{2t} + Q_\ell \right)^\theta \right| < \frac{1}{2},
		\end{equation*}
or equivalently, 
		\begin{equation} \label{eq2.3.22}
			Q_\ell^\theta \left|  \left( 1 + \frac{y_1}{Q_\ell} \right)^\theta + \cdots + \left( 1 + \frac{y_t}{Q_\ell} \right)^\theta - \left( 1 + \frac{y_{t+1}}{Q_\ell} \right)^\theta - \cdots - \left( 1 + \frac{y_{2t}}{Q_\ell} \right)^\theta  \right|< \frac{1}{2}.
		\end{equation}			

We consider the function $h : (-1,1) \to \mathbb R $ with $ h(z) = (1+z)^\theta.$ Then a Taylor expansion up to the $k = \lfloor 2 \theta \rfloor + 1$ term around the point $z_0 = 0$ yields
	\begin{equation*}	
			h(z) =  h(0) + \sum_{j=1}^k \frac{h^{(j)}(0)}{ j \,!} z^j + r_k(z) =1 + \sum_{ j=1}^k b_j z^j + r_k(z),
	\end{equation*}
where recall that 
	\begin{equation*}
		b_j = \binom{ \theta}{j} =  \frac{\theta (\theta -1) \cdots (\theta-j+1) }{j \,!}, 
	\end{equation*} 
is the $j$-th combinatorial coefficient of the expansion. Here $r_k(z)$ denotes the remainder term. In Lagrange's form the remainder term takes the shape
	\begin{equation} \label{eq2.3.23}
		r_k(z) = \frac{h^{(k+1)}(c)}{(k+1) \,!} z^{k+1} =  b_{k+1} ( 1 +  c)^{\theta -k -1} z^{k+1}, 
	\end{equation}
for some $c$ between $0$ and $z.$ 

For each index $ 1 \leq i \leq 2t$ we write $ z_i = y_i / Q_\ell.$ In view of (\ref{eq2.3.21}) and for sufficiently large $P$ one has 
	\begin{equation*} 
		0 \leq   \frac{ P_\ell - \lfloor P_\ell \rfloor }{ \lfloor P_\ell \rfloor} < z_i < \frac{2}{ \sqrt P} < 1.
	\end{equation*}
Indeed, this follows immediately upon writing
	\begin{equation*}
		z_i = \frac{y_i}{Q_\ell} \leq \frac{ \sqrt P +1}{ \lfloor P_\ell \rfloor} <  \frac{ \sqrt P +1}{P_\ell -1} < \frac{2}{ \sqrt P }.
	\end{equation*}
Thus, using (\ref{eq2.3.23}) with $ z = z_i$ and upon noting that $ 1 + c > 1$ and $ \theta - k -1 < 0,$ we may infer the following upper bound for the error term
	\begin{equation} \label{eq2.3.24}
		\begin{split}	
			| r_k(z_i) |  &\leq | b_{k+1} |  \left( \frac{2} { \sqrt P} \right)^{k+1} \\					
			& = | b_{k+1} | 2^{k+1} P^{- \frac{k+1}{2}}.
		\end{split}
	\end{equation}
		
Expanding each term occurring in (\ref{eq2.3.22}), we obtain that	
	\begin{equation} \label{eq2.3.25}
		\begin{split}	
			Q_\ell^\theta \left( 1 + \frac{y_i}{Q_\ell} \right)^\theta & = Q_\ell^\theta \left( 1 + b_1 \left(\frac{y_i}{Q_\ell} \right) + \cdots + b_k \left( \frac{y_i}{Q_\ell} \right)^k + r_k \left ( \frac{ y_i}{ Q_\ell} \right) \right) \\[10pt]
			& = Q_\ell^\theta + b_1 Q_\ell^{\theta -1} y_1 + \cdots + b_k	Q_\ell^{\theta -k} y_i^k + Q_\ell^\theta r_k \left ( \frac{ y_i}{ Q_\ell} \right).				
		\end{split}
	\end{equation}
For large $P$ one has $ Q_\ell \leq 2 P. $ So by (\ref{eq2.3.24}) and since $ k+1 = \lfloor 2 \theta \rfloor +2 > 2 \theta ,$ we infer that as $ P \to \infty $ one has
	\begin{equation*}
		\begin{split}
			\left| Q_\ell^\theta r_k \left ( \frac{ y_i}{ Q_\ell} \right) \right| & \leq | b_{k+1} | 2^{k+1}  P^{- \frac{k+1}{2} } \left( 2  P \right)^\theta \\[10pt]
			& =  | b_{k+1} | 2^{k+1+\theta}  P^{\theta - \frac{k+1}{2}} \\[10pt]
			& = o (1).
		\end{split}
	\end{equation*}
Consequently, when $P$ is large enough in terms of $k$ one has for each index $ 1 \leq i \leq 2t $  that
		\begin{equation} \label{eq2.3.26}
			\left| Q_\ell^\theta r_k \left ( \frac{ y_i}{ Q_\ell} \right) \right| \leq \frac{1}{4t}.
		\end{equation}
		
Substituting the asymptotic expansion (\ref{eq2.3.25}) into (\ref{eq2.3.22}) and taking into account (\ref{eq2.3.26}) together with the symmetry of the inequality, we deduce that the number of integer solutions of the inequality (\ref{eq2.3.22}) is bounded above by the number of integer solutions of the inequality
		\begin{equation*}
			\left| \sum_{j=1}^k b_j Q_{\ell}^{\theta -j} \left( y_1^j + \cdots + y_t^j -  y_{t+1}^j - \cdots  - y_{2t}^j \right) \right|  < \frac{2t}{ 4t} + \frac{1}{2} = 1.
		\end{equation*}
Rearranging the terms in the summation on the left hand side of the above expression, we can rewrite the last inequality in an equivalent form as	
		\begin{equation}  \label{eq2.3.27}	
			\left| b_1 Q_\ell^{\theta -1} \sum_{i=1}^t (y_i - y_{t+i}) + \cdots + b_k Q_\ell^{ \theta -k} \sum_{i=1}^t (y_i^k - y_{t+i}^k)  \right| < 1.
		\end{equation}

The number of integer solutions of the inequality (\ref{eq2.3.27}) with $ 0 < y_i < 1 + \lfloor \sqrt P \rfloor$ is bounded above by the number of integer solutions of the system
		\begin{equation} \label{eq2.3.28}
			\begin{cases}
				\displaystyle \left| b_1 Q_\ell^{\theta -1} h_1 + \cdots + b_k Q_\ell^{\theta-k} h_k \right| < 1 \\[15pt]
				\displaystyle \sum_{i=1}^t (y_i^j - y_{t+i}^j) = h_j \hspace{0.1in} (1 \leq j \leq k)								
			\end{cases}
		\end{equation}
with $ 0 < y_i \leq Y$ where $Y = 1+ \lfloor \sqrt P \rfloor .$ We denote this counting function by $ Z_{t,k} \left( Y ; \tupleh \right). $ Note that the integers $ h_j$  satisfy the relation $ 	| h_j | \leq t Y^j  \hspace{0.1in} (1 \leq j \leq k).$	

We write $ J_{t, k} \left(  Y; \tupleh \right) $ to denote the number of integer solutions of the inhomogeneous Vinogradov system 
		\begin{equation*}
			\sum_{i=1}^t ( y_i^j - y_{t+i}^j ) = h_j \hspace{0.5in} (1 \leq j \leq k),
		\end{equation*}
with $ 0 < y_i  \leq Y.$ By orthogonality one has
		\begin{equation*}
			J_{t, k} \left(  Y ; \tupleh \right) = \int_{ [0,1)^k} \left| \sum_{0 < y_i \leq Y} e ( \alpha_1 y + \cdots + \alpha_k y^k) \right|^{2t} e( - \tuplealpha \cdot \tupleh) \text d \tuplealpha ,
		\end{equation*}
where as usual $ \tuplealpha \cdot \tupleh$ stands for $ \alpha_1 h_1 + \cdots + \alpha_k h_k.$ By the triangle inequality and in view of Theorem \ref{thm2.3.1} one has for any fixed $ \epsilon > 0$ that 
		\begin{equation} \label{eq2.3.29}
			J_{t, k} \left(  Y ; \tupleh \right) \leq	J_{t, k} \left(  Y  \right) \ll Y^{2t - \frac{1}{2} k(k+1) + \epsilon}.
		\end{equation}

Recall the definition (\ref{eq2.3.4}) of the expression $ \mathcal H ,$ where $P$ is now replaced by $Q_\ell.$ Turning our attention to the system (\ref{eq2.3.28}) we see that 
		\begin{equation*}
			Z_{t,k} \left( Y ; \tupleh \right) \ll  \sum_{ \substack{ | h_j | \leq t Y^j  \\ 1 \leq j \leq k \\ | \mathcal H (\tupleh) | < 1 }} J_{t,k} \left( Y ; \tupleh \right), 	
		\end{equation*} 		
which by the triangle inequality leads to 
		\begin{equation} \label{eq2.3.30}
			\begin{split}
				Z_{t,k} \left( Y ; \tupleh \right)  \ll  J_{t,k} \left( Y \right)  \sum_{ \substack{ | h_j | \leq t Y^j  \\ 1 \leq j \leq k \\ | \mathcal H (\tupleh) | < 1 }} 1.
			\end{split}
		\end{equation}
Recall now that $Y = 1+ \lfloor \sqrt P \rfloor \ll \sqrt P \asymp Q_\ell^{1/2}.$ One may estimate the sum on the right hand side of (\ref{eq2.3.30}) by invoking Lemma \ref{lemma2.3.3}.  Hence, appealing to (\ref{eq2.3.29}) and Lemma \ref{lemma2.3.3} the estimate (\ref{eq2.3.30}) now delivers
		\begin{equation} \label{eq2.3.31}
			\begin{split}	
				Z_{t,k} \left( Y ; \tupleh \right) & \ll Y^{2t - \frac{1}{2} k(k+1) + \epsilon} \cdot P^{ \frac{1}{4}k(k+1) - \theta + \frac{1}{2}} \\[10pt]
				& \ll P^{ t - \theta + \frac{1}{2} + \epsilon}.
			\end{split}
		\end{equation}					

Putting together (\ref{eq2.3.31}), (\ref{eq.2.3.19}) and (\ref{eq2.3.18}) we deduce that
	\begin{equation*}
		V_t \left( I ; \frac{1}{2} \right) \ll P^{2t - \theta + \epsilon},
	\end{equation*}
which in view of (\ref{eq2.3.7}) completes the proof of the theorem.		
\end{proof}
	
It is convenient for the rest of the analysis to have in hand an estimate for the exponential sum $ f_i (\alpha) = f(\lambda_i \alpha).$ 

	\begin{corollary} \label{cor2.3.5}
		Let $ \lambda$ be a fixed real number. Suppose that $ \kappa $ is a real number such that $ \kappa | \lambda | \geq 1 .$ Suppose further that $ t \geq  \frac{1}{2} \left( \lfloor 2\theta  \rfloor +1 \right) \left( \lfloor 2\theta  \rfloor +2 \right) $ is a natural number. Then for any fixed $ \epsilon > 0$ one has that			
			\begin{equation*}
				\int_{-\kappa}^{ \kappa} \left| f ( \lambda \alpha)  \right|^{2t} \normalfont \text d \alpha \ll \kappa P^{2t - \theta + \epsilon}.
			\end{equation*}		
		The implicit constant on the above estimate depends on $\epsilon, \lambda, t,$and $\theta,$ but not on $ \kappa$ and $P.$
	\end{corollary}
	
	\begin{proof}
		Using the fact that $ f(-\alpha) = \overline{ f(\alpha)}$ for all $ \alpha \in \mathbb R $ and changing variables, we see that
			\begin{equation*}
				\begin{split}
					\int_{ 0}^{  \kappa} \left| f( \lambda \alpha ) \right|^{2t} \text d \alpha = \int_{ 0}^{  \kappa} \left| f( | \lambda | \alpha) \right|^{2t} \text d \alpha & = \frac{1}{ | \lambda |} \int_{0}^{  \kappa | \lambda |} \left| f(u) \right|^{2t} \text d u \\[10pt]
					& \ll_{\lambda} \int_{ -  \kappa | \lambda | }^{  \kappa | \lambda |} \left| f(u) \right|^{2t} \text d u.
				\end{split}
			\end{equation*}		
		Invoking Theorem \ref{thm2.1.4} we are done.
	\end{proof}
	
%%%%% END OF APPROXIMATELY TDI SYSTEMS %%%%%%%%%%%%%%%

%%%%% BEGINNNING OF MINOR ARCS ANALYSIS %%%%%%%%%%%%%
	
\section{Minor arcs analysis} \label{section_minor_arcs}

We begin the analysis of the analytical representation (\ref{eq2.2.11}) with the contribution coming from the minor arcs. Recall that this set is given by 
	\begin{equation*}
		\mathfrak{m} = \left\{ \alpha \in \mathbb R : P^{-\theta + \delta_0} \leq | \alpha | < P^{ \omega } \right\} .
	\end{equation*}
Define the intervals $ \mathfrak m^{+} = \left[ P^{- \theta +\delta_0}, P^{\omega } \right) $ and $ \mathfrak m^{-} = \left( -P^{\omega }, - P^{- \theta +\delta_0} \right] $ and note that $ \mathfrak m = \mathfrak m^{+}  \cup \mathfrak m^{-}.$ One has  $ f_i(- \alpha) = \overline{ f_i(\alpha)}$ for all $ \alpha \in \mathbb R .$ Moreover, the kernel functions $K_\pm (\alpha)$ are real valued and even. Recall (\ref{eq2.2.10}). By a change of variables one has
	\begin{equation}  \label{eq2.4.1}
		R_\pm \left(P ;\mathfrak m^{-} \right) = \overline{ R_\pm \left( P ;\mathfrak m^{+} \right)},
	\end{equation} 
where $ \overline{ R_\pm (P ; \mathfrak m^{+})}$ stands for the complex conjugate. Therefore it suffices to deal with the set $ \mathfrak m^{+}.$

We make use of the following variant of van der Corput's $k$-th derivative test, for bounding exponential sums.

	\begin{lemma} \label{lemma2.4.1}
		Let $q \geq 0$ be an integer. Suppose that $ f : (X, 2X] \to \mathbb R$ is a function having continuous derivatives up to the $(q+2)$-th order in $(X,2X].$ Suppose also there is some $F > 0,$  such that for all $ x \in (X, 2X]$ we have
			\begin{equation} \label{eq2.4.2}
				F X^{-r} \ll | f^{(r)} (x) | \ll FX^{-r},
			\end{equation}
		for $r =1, 2, \ldots, q+2.$ Then we have
			\begin{equation*}
				\sum_{X < x \leq 2X} e(f(x)) \ll F^{1/ (2^{q+2}-2)} X^{1 - (q+2)/(2^{q+2}-2)} + F^{-1}X,
			\end{equation*}
		with the implied constant depending only upon the implied constants in (\ref{eq2.4.2}).
	\end{lemma}

	\begin{proof}
		See \cite[Theorem 2.9]{graham_kolesnik_book}.
	\end{proof}
	
Recall that in (\ref{eq2.3.5}) we defined the exponential sum $g(\alpha) = g(\alpha ; P)$ by
	\begin{equation*}
		g(\alpha ; P) = \sum_{P < x \leq 2P } e(\alpha x^\theta).
	\end{equation*}
We put $ g_i (\alpha ) = g( \lambda_i \alpha ) \hspace{0.05in} (1\leq  i\leq s).$ Below we give a crude non-trivial upper bound for the exponential sum $f_i(\alpha)$ when $ \alpha \in  \mathfrak m^{+}.$ One can certainly improve this estimate. However, for our purposes the saving we obtain is sufficient.
	
	\begin{lemma} \label{lemma2.4.2}
		For each index $ 1 \leq  i \leq s $ one has for any fixed $ \epsilon > 0$ that
			\begin{equation} \label{eq2.4.3}
				\sup_{ \alpha \in \mathfrak m^{+}} \left| f_i (\alpha) \right| \ll P^{1 - 4^{-\theta} + \epsilon}.
			\end{equation}
	\end{lemma}

	\begin{proof}
		Fix an index $i.$ It suffices to show that
			\begin{equation*}
				\sup_{ \alpha \in \mathfrak m^{+}} \left| g_i (\alpha) \right| \ll P^{1 - 4^{-\theta}}.
			\end{equation*}
		Then one may split the interval $[1,P]$ into $ O \left( \log P \right)$ dyadic intervals and the desired conclusion follows.  
		
		We set $ \phi (x) = \lambda_i \alpha x^\theta.$ For each integer $ r \geq 1$  one has $\phi^{(r)} (x) = C_r \alpha x^{\theta -r},$ where we put $ C_r = \lambda_i  \theta (\theta -1) \cdots  (\theta - r+1).$ It is apparent that for $ P  < x \leq 2P$ one has
			\begin{equation*}
				\left| 	\phi^{(r)} (x) \right| \asymp F P^{-r},
			\end{equation*}
		where $ F = | C_r| | \alpha| P^\theta.$ For $ \alpha \in \mathfrak m^{+} = \left[ P^{-\theta + \delta_0}, P^\omega \right)$ one has
			\begin{equation*}
				 | C_r| P^{\delta_0} \leq F < |C_r| P^{\theta + \omega}.
			\end{equation*}
		
		We apply Lemma \ref{lemma2.4.1} with $ q =n,$ where $ n = \lfloor \theta \rfloor$ is the integer part of $ \theta.$ This yields that for any $ \alpha \in \mathfrak m^{+}$ one has
			\begin{equation*}
				\left| g_i (\alpha ; X) \right| \ll_{\lambda_i, \theta} P^{ 1 - \eta} + P^{1 - \delta_0},
			\end{equation*}
		where 
			\begin{equation*}
				\eta = \frac{n+2 - \theta - \omega}{2^{n+2}-2}.
			\end{equation*}
		Upon recalling (\ref{eq2.2.9}) one may easily verify that $ \eta > 4^{-\theta}$ which completes the proof.
	\end{proof}

By (\ref{eq2.2.4}) one has $|K_\pm (\alpha) | \ll 1.$ An application of H{\"o}lder's inequality reveals that
	\begin{equation*}				
		\int_{\mathfrak m^{+}}  | f_1 (\alpha) \cdots f_s( \alpha)  K_{\pm} (\alpha) | \text{d} \alpha \ll
		\left( \int_{\mathfrak m^{+}} | f_1 (\alpha) |^s \text{d} \alpha \right)^{1/s} \cdots \left( \int_{ \mathfrak m^{+}} | f_s (\alpha) |^s \text{d} \alpha  \right)^{1/s}.   			
	\end{equation*}	
We set $ \kappa = P^\omega.$ Note that for large enough $P$ one has $ P^\omega | \lambda_i | \geq 1.$ Combining Corollary \ref{cor2.3.5} and the upper bound recorded in (\ref{eq2.4.3}) we deduce that for any fixed $ \epsilon > 0$ one has
	\begin{equation*}
		\begin{split}
				\int_{ \mathfrak m^{+} }	| f_i (\alpha) |^s \text{d} \alpha  & \ll \left( \sup_{\alpha \in \mathfrak m^{+} }  | f_i (\alpha) | \right)^{s-2t} \int_{-P^{\omega}}^{P^\omega}  | f_i (\alpha) |^{2t} \text{d} \alpha \\[10pt]
				& \ll P^{s - \theta} \cdot P^{- 4^{-\theta}(s-2t) + \omega + \epsilon(s-2t+1)},
		\end{split}
	\end{equation*}
provided that $ s > 2t \geq (\lfloor 2 \theta \rfloor + 1) (\lfloor 2 \theta \rfloor + 2) .$ Choosing  $ \epsilon = 5^{-100\theta} > 0$ as we are at liberty to do and recalling from (\ref{eq2.2.9}) that $ \omega \leq 5^{-100 \theta},$ we infer that
	\begin{equation*}
	 	\int_{ \mathfrak m^{+} }	| f_i (\alpha) |^s \text{d} \alpha  \ll P^{s- \theta} \cdot P^{ - 4^{-\theta}(s-2t) + 5^{-100\theta}(s- 2t +2)} = o \left( P^{s - \theta} \right).
	\end{equation*}
In the light of (\ref{eq2.4.1}) we have established the following.
	
	\begin{lemma} \label{lemma2.4.3}
		One has
			\begin{equation*}
				\int_{\mathfrak m} | f_1(\alpha) \cdots f_s(\alpha) K_\pm(\alpha) | \normalfont\text{d} \alpha = o \left( P^{s-\theta} \right),
			\end{equation*}			
		provided $ s \geq (  \lfloor 2 \theta \rfloor + 1) (  \lfloor 2 \theta \rfloor + 2) +1.$ 	
	\end{lemma}

%%%%% END OF MINOR ARCS ANALYSIS %%%%%%%

%%%%% BEGINNING OF TRIVIAL ARCS ANALYSIS %%%%%%
	
\section{Trivial arcs analysis} \label{section_complementary_arcs}

In this section we deal with the set of trivial arcs. Recall that this set is given by
	\begin{equation*}
		\mathfrak t = \{ \alpha \in \mathbb R : | \alpha | \geq P^{ \omega}  \}. 
	\end{equation*}
Define $ \mathfrak t^{+} = [P^\omega, \infty)$ and $ \mathfrak t^{-} = (-\infty, -P^\omega]$ and note that $ \mathfrak t =  \mathfrak t^{+} \cup \mathfrak t^{-}.$ Recall (\ref{eq2.2.10}). A change of variables as in section \ref{section_minor_arcs} yields
	\begin{equation} \label{eq2.5.1}
		R_\pm \left( P ;\mathfrak t^{-} \right) = \overline{ R_\pm \left( P ;\mathfrak t^{+} \right)}.
	\end{equation}
Hence it suffices to deal with the set $ \mathfrak t^{+}.$ By (\ref{eq2.2.4}) with $h=1$ one has that
	\begin{equation*} 
		\int_{\mathfrak t^{+}} | f_1(\alpha) \cdots f_s(\alpha)  K_\pm(\alpha)| \text d \alpha  \ll \sum_{ j = \lfloor \omega  \log_2 P \rfloor }^\infty \frac{ \left(\log P \right)}{ 2^{2j} } \int_{2^j}^{ 2^{j+1}}  | f_1(\alpha) \cdots f_s(\alpha) | \text d \alpha.
	\end{equation*}
	
An application of H\"{o}lder's inequality yields
	\begin{equation} \label{eq2.5.2}
		\int_{2^j}^{ 2^{j+1}} | f_1 (\alpha) \cdots f_s (\alpha) | \text{d} \alpha	\ll \left( \prod_{i=1}^s \int_{2^j}^{ 2^{j+1}} | f_i ( \alpha)|^s \text{d} \alpha \right)^{1/s}. 
	\end{equation}
Define $ s_0 = (\lfloor 2 \theta\rfloor +1)( \lfloor 2 \theta \rfloor +2)$ and note that $s_0$ is even. In making the trivial estimate
	\begin{equation*}
		\left| f_i (\alpha) \right| = O \left( P \right)
	\end{equation*}
it follows that
	\begin{equation*}
		\int_{ 2^j }^{ 2^{j+1} }  | f_i ( \alpha)|^s \text d \alpha \ll	 P^{s-s_0} 	\int_{ 2^j }^{ 2^{j+1} }  | f_i ( \alpha)|^{s_0} \text d \alpha,
	\end{equation*}
provided that $ s \geq s_0.$ For sufficiently large $P$ and for $  j \geq \lfloor \omega  \log_2 P \rfloor +1$ one has $ 2^{j+1} | \lambda_i| \geq 1 $ for each index $i.$ Invoking Corollary \ref{cor2.3.5} the above estimate yields that for any fixed $ \epsilon > 0$ one has
	\begin{equation*}
		\int_{ 2^j }^{ 2^{j+1} }  | f_i ( \alpha)|^s \text d \alpha \ll	 2^{j+1} P^{s-\theta + \epsilon} \hspace{0.5in} (1 \leq i \leq s).
	\end{equation*}
 By (\ref{eq2.5.2}) we infer that 
	\begin{equation*}
		\int_{\mathfrak t^{+}} | f_1(\alpha) \cdots f_s(\alpha)  K_\pm(\alpha)| \text d \alpha  \ll P^{ s- \theta + \epsilon}   \sum_{ j = \lfloor \omega \log_2 P \rfloor  }^\infty \frac{1}{ 2^{j }}.
	\end{equation*}

Clearly one has
	\begin{equation*}
		\sum_{ j = \lfloor \omega \log_2 P \rfloor  }^\infty \frac{1}{ 2^{j }} \ll P^{- \omega}.
	\end{equation*}
Hence by choosing $ \epsilon = \frac{\omega}{2} > 0$ the previous estimate delivers
		\begin{equation*}
			\int_{\mathfrak t^{+}} | f_1(\alpha) \cdots f_s(\alpha)  K_\pm(\alpha)| \text d \alpha  \ll P^{s- \theta  - \frac{\omega}{2} } =  o \left( P^{s-\theta} \right).
		\end{equation*}
In the light of (\ref{eq2.5.1}) we have established the following.

	\begin{lemma} \label{lemma2.5.1}
		One has
			\begin{equation*}
				\int_{\mathfrak t} | f_1(\alpha) \cdots f_s(\alpha) K_\pm(\alpha) | \normalfont  \text d \alpha = o\left( P^{s-\theta}\right),
			\end{equation*}
		provided $ s \geq  (  \lfloor 2 \theta \rfloor +1) (  \lfloor 2 \theta \rfloor +2).$
	\end{lemma}
	
%%%%%% END OF TRIVIAL ARCS ANALYSIS %%%%%%%%%%%
	
%%%%%% BEGINNING OF MAJOR ARC ANALYSIS %%%%%%%%%	

\section{Major arc analysis and the asymptotic formula}  \label{section_major_arc_asympt_formular}

Now we deal with the contribution of the major arc
	\begin{equation*}
		\mathfrak M = \{ \alpha \in \mathbb R : |\alpha| < P^{- \theta + \delta_0} \}
	\end{equation*}
around zero. The corresponding analytical approximations for the generating functions  $f_i( \alpha) $ are given by
	\begin{equation} \label{eq2.6.1}
		\upsilon_i (\alpha) = \upsilon(\lambda_i \alpha) = \int_0^P e(\lambda_i \alpha \gamma^\theta) \text d \gamma \hspace{0.5in} (1\leq i \leq s).
	\end{equation}
An application of partial summation delivers 
	\begin{equation*} \label{major_arc_difference_f_upsilon}
		f_i( \alpha) - \upsilon_i(\alpha) = O\left( 1 + P^\theta | \alpha | \right),
	\end{equation*}
uniformly for $ \alpha \in \mathbb R.$ Thus for $ \alpha \in \mathfrak M$ one has
	\begin{equation*}
		f_i (\alpha) - \upsilon_i (\alpha) \ll P^{\delta_0}.
	\end{equation*}
The above estimate in combination with the trivial bounds $|f_i(\alpha)|, |\upsilon_i(\alpha)| \leq P$ and the telescoping sum 
	\begin{equation*}
		f_1 \cdots f_s - \upsilon_1 \cdots \upsilon_s = \sum_{i=1}^s f_1 \cdots f_{i-1} (f_i - \upsilon_i) \upsilon_{i+1} \cdots \upsilon_s,
	\end{equation*}
reveals that for $ \alpha \in \mathfrak M$ one has 
	\begin{equation*}
		f_1(\alpha) \cdots  f_s(\alpha) - \upsilon_1(\alpha) \cdots \upsilon_s(\alpha)  \ll  O(P^{s - 1 + \delta_0}).
	\end{equation*}

Integrating over $ \mathfrak M$ yields
	\begin{equation} \label{eq2.6.2}
		\begin{split}
			\int_{\mathfrak M} f_1( \alpha) \cdots f_s( \alpha)  K_{\pm}( \alpha) \text{d} \alpha -  \int_{\mathfrak M} \upsilon_1(\alpha) \cdots \upsilon_s(\alpha) K_{\pm}(\alpha) \text{d} \alpha 
			& \ll \int_{\mathfrak M} P^{s-1 + \delta_0} \text{d} \alpha \\
			& = P^{s-1 + \delta_0} \text{meas}\left(\mathfrak M \right) \\
			& \asymp P^{s- \theta - 1 + 2\delta_0},								
		\end{split}
	\end{equation}
where in the last step we used the fact $ \text{meas}\left( \mathfrak M\right) \asymp P^{-\theta + \delta_0}.$ By (\ref{eq2.2.4}) one has $ | K_{\pm} (\alpha)| \ll 1.$  Using integration by parts one has 
	\begin{equation} \label{eq2.6.3}
		\upsilon_i(\alpha) \ll_{\lambda_i} \min\{ P, | \alpha|^{- 1/ \theta}\} \ll_{\lambda_i} \frac{P}{\left( 1+ P^\theta | \alpha| \right)^{1/\theta}} \hspace{0.2in} (1 \leq i \leq s).
	\end{equation}
So we deduce that
	\begin{equation} \label{eq2.6.4}
		\int_{ \mathbb R \setminus \mathfrak M} \upsilon_1 (\alpha) \cdots \upsilon_s (\alpha) K_{\pm} ( \alpha) \text{d} \alpha   \ll \int_{ | \alpha | > P^{-\theta + \delta_0} } | \alpha |^{- s/ \theta} \text{d} \alpha \ll P^{s- \theta - \delta_0( s / \theta -1)},
	\end{equation}
where in the last step we used the hypothesis $ s > 2 \theta.$

The singular integral of our problem is given by
	\begin{equation} \label{eq.2.6.5}
		\mathcal I_{\pm} = \int_{-\infty}^\infty \upsilon_1(\alpha) \cdots \upsilon_s(\alpha) K_{\pm}(\alpha) \text{d} \alpha.
	\end{equation}	
Note that by (\ref{eq2.2.4}) and (\ref{eq2.6.3}) the integral $ \mathcal I_\pm $ is well defined and absolutely convergent. Combining (\ref{eq2.6.2}) and (\ref{eq2.6.4}) and since $ 2 \delta_0 < 1,$  we see that 
	\begin{equation} \label{eq2.6.6}
		\mathcal I_{\pm} = \int_{\mathfrak M} f_1(\alpha) \cdots f_s(\alpha) K_{\pm}(\alpha) \normalfont\text{d} \alpha   + o \left( P^{s -\theta}  \right).
	\end{equation}	

For $ \alpha \in \mathbb R $ we put
	\begin{equation} \label{eq2.6.7}
		\Phi( \alpha) = \upsilon_1(\alpha) \cdots \upsilon_s(\alpha) = \int_{ [0,P]^s} e \left( \alpha ( \lambda_1 \gamma_1^\theta + \cdots +\lambda_s \gamma_s^\theta)\right) \text d  \tuplegamma.
	\end{equation}
In view of (\ref{eq2.6.3}) we see that $ \Phi$ is an integrable function. Making a change of variables by putting $ \gamma_i = P \left( \beta_i | \lambda_i|^{-1} \right)^{ 1 / \theta} \hspace{0.05in} (1 \leq i \leq s)$ yields
	\begin{equation} \label{eq2.6.8}
		\Phi(\alpha) =\left(\frac{P}{\theta} \right)^s | \lambda_1 \cdots \lambda_s|^{-1/\theta} \int_{\mathcal B}  ( \beta_1 \cdots \beta_s)^{1/\theta -1} e\left( \alpha P^\theta( \sigma_1 \beta_1 + \cdots + \sigma_s \beta_s)\right) \text d  \tuplebeta,
	\end{equation}
where $ \sigma_i = \lambda_i / | \lambda_i| \in \{\pm 1\} $ are not all equal and $ \mathcal B = [ 0, | \lambda_1|] \times \cdots \times [0, |\lambda_s|].$  Let $ \widetilde \beta \in \mathbb R $ be a parameter. We now write $ \hspace{0.05in} \mathcal U ( \widetilde \beta) = \mathcal U (\widetilde \beta ; \tuplelambda) \subset \mathbb R^{s-1}$ for the domain defined through the linear inequalities
	\begin{equation*}
		0 \leq \beta_i \leq | \lambda_i | \hspace{0.1in} (1\leq i \leq s-1), \hspace{0.3in} 0  \leq \widetilde \beta - \sigma_s \sigma_1 \beta_1 - \cdots - \sigma_s \sigma_{s-1} \beta_{s-1} \leq  | \lambda_s |.
	\end{equation*} 
We set
	\begin{equation*}
		\Psi_0( \widetilde \beta) = \int_{\mathcal U(\widetilde \beta) } \left( \widetilde \beta -  \sigma_s \sigma_1 \beta_1 - \cdots - \sigma_s \sigma_{s-1} \beta_{s-1} \right)^{1/ \theta -1} ( \beta_1 \cdots \beta_{s-1})^{1/ \theta -1} \text d \beta_1 \cdots \text d \beta_{s-1}.
	\end{equation*}
For $ \widetilde \beta \in [0, | \lambda_s|]$ the map $ \widetilde \beta \mapsto \Psi_0( \widetilde \beta)$ defines a non-negative and continuous function. Put
	\begin{equation} \label{eq2.6.9}
		\Psi ( \widetilde \beta) = 
				\begin{cases}
					\Psi_0( \widetilde \beta), \hspace{0.1in} & \text{ if} \hspace{0.1in} \widetilde \beta  \in [0, | \lambda_s |], \\[10pt]
					0,  \hspace{0.1in} & \text{otherwise}.
				\end{cases}
	\end{equation}
Note that $ \Psi ( \widetilde \beta) $ is a non-negative and compactly supported function defined over $ \mathbb R ,$ which has precisely two points of discontinuity, at $ \widetilde \beta = 0, | \lambda_s |.$ We set 
	\begin{equation*} 
		\widetilde \beta = \beta_s + \sigma_1 \sigma_s \beta_1 + \cdots + \sigma_{s-1} \sigma_s \beta_{s-1}.
	\end{equation*}
Replace in (\ref{eq2.6.8}) the variable $ \beta_s $ by $ \widetilde \beta .$ Letting now $ \widetilde \beta $ vary through $ \mathbb R $ and using the fact that $ \Psi( \widetilde \beta)$ is compactly supported we obtain
	\begin{equation} \label{eq2.6.10}
		\Phi(\alpha) =\left(\frac{P}{\theta} \right)^s | \lambda_1 \cdots \lambda_s|^{-1/\theta} \int_{- \infty}^{\infty} \Psi( \widetilde \beta) e ( \alpha P^\theta \sigma_s \widetilde \beta) \text d  \widetilde \beta.
	\end{equation}
Since $ \Phi$ and $ \Psi$ are integrable we may apply Fourier's inversion theorem. Together with a substitution that replaces $ \alpha $ by $ \alpha P^{- \theta}$  we obtain that
	\begin{equation} \label{eq2.6.11}
		\Psi ( \widetilde \beta) = \left( \frac{\theta}{P} \right)^s | \lambda_1 \cdots \lambda_s |^{1/ \theta} \int_{-\infty}^{\infty} \Phi( \alpha P^{-\theta}) e( - \sigma_s \widetilde \beta \alpha) \text d \alpha. 
	\end{equation}
Putting together (\ref{eq.2.6.5}), (\ref{eq2.6.8}) and (\ref{eq2.6.10}), we infer that
	\begin{equation} \label{eq2.6.12}
		\begin{split}
			\mathcal I_{\pm} & = \left(\frac{P}{\theta} \right)^s | \lambda_1 \cdots \lambda_s|^{-1/\theta} \int_{- \infty}^\infty \int_{-\infty}^{\infty} \Psi (\widetilde \beta) e( \alpha P^\theta \sigma_s \widetilde \beta) K_\pm (\alpha) \text d \alpha \text d \widetilde \beta \\[15pt]
			& = \left(\frac{P}{\theta} \right)^s | \lambda_1 \cdots \lambda_s|^{-1/\theta} 	\int_{- \infty}^\infty 	\Psi (\widetilde \beta)  \left( \int_{- \infty}^\infty  e( \alpha P^\theta \sigma_s \widetilde \beta) K_\pm (\alpha) \text d \alpha \right) \text d \widetilde{\beta}.
		\end{split}
	\end{equation}

By the comment following (\ref{eq2.2.7}) one has  
	\begin{equation} \label{eq2.6.13}
		\int_{-\infty}^\infty e( \alpha P^\theta \sigma_s \widetilde \beta) K_\pm (\alpha) \text d \alpha = \chi_\tau \left( P^\theta \sigma_s \widetilde \beta \right),
	\end{equation}
unless $ \widetilde \beta $ satisfies the relation $ \left| |  P^\theta \sigma_s \widetilde \beta | - \tau \right| < \widetilde \tau,$ where recall that we have set $ \widetilde \tau = \tau \left( \log  P \right)^{-1}.$ The measure of the set of points $ \widetilde \beta $ which satisfy the latter inequality is $ O \left( \tau P^{-\theta} \left( \log P \right)^{-1} \right).$ The contribution coming from this set of $\widetilde \beta$ is $o \left(P^{s-\theta}\right),$ and hence one may ignore this set. Therefore, we may assume from now on that (\ref{eq2.6.13}) is valid. Recalling that $ \chi_\tau ( \cdot) $ denotes the characteristic function of the interval $(- \tau, \tau),$ we may rewrite  (\ref{eq2.6.13}) as
	\begin{equation} \label{eq2.6.14}
		\int_{-\infty}^\infty e( \alpha P^\theta \sigma_s \widetilde \beta) K_\pm (\alpha) \text d \alpha =
			\begin{cases}
				1,  & \text{if} \hspace{0.05in} | \widetilde \beta | < \tau P^{-\theta}, \\[10pt]
				0, & \text{otherwise}.				
			\end{cases}
	\end{equation}
	
	\begin{lemma} \label{lemma2.6.1}
		For $ | \widetilde \beta | < \tau P^{-\theta} $ one has 
		\begin{equation*}
			\Psi( \widetilde \beta) - \Psi(0) \ll \tau P^{-\theta}. 
		\end{equation*}
	\end{lemma}
	
	\begin{proof}
		By (\ref{eq2.6.7}) and (\ref{eq2.6.3}) one has
		\begin{equation*}
			| \Phi( \alpha P^{-\theta})| = \prod_{i=1}^s \left| \upsilon_i (\alpha P^{-\theta}) \right| \ll \frac{P^s}{ (1 + | \alpha|)^{s/ \theta}}.
		\end{equation*}
		Thus by (\ref{eq2.6.11}) one has
			\begin{equation*}
				\begin{split}
					| \Psi( \widetilde \beta) - \Psi(0) | & \leq  \left( \frac{\theta}{P} \right)^s | \lambda_1 \cdots \lambda_s |^{-1/ \theta} \int_{- \infty}^{\infty} | \Phi(\alpha P^{-\theta}) | | e( - \sigma_s \widetilde \beta \alpha) - 1 | \text d \alpha \\[10pt]
					& \ll 	\int_{- \infty}^{\infty} \frac{1}{ (1 + | \alpha|)^{s / \theta}}  | e( - \sigma_s \widetilde \beta \alpha) - 1 | \text d \alpha.				 			
				\end{split}
			\end{equation*}
		Note that for any $ x \in \mathbb R $ one has
			\begin{equation*}
				\left| e (x) - 1 \right| \leq 2 \pi |x|.
			\end{equation*}
		Using this inequality we deduce that
			\begin{equation*}
					\Psi( \widetilde \beta) - \Psi(0) \ll | \widetilde \beta | 	\int_{- \infty}^{\infty} \frac{\left| \alpha\right|}{ (1 + \left| \alpha \right|)^{s / \theta}} 	\text d \alpha \ll \tau P^{-\theta},	
			\end{equation*} 
		since for $ s > 2 \theta $ the integral with respect to $ \alpha$ is absolutely convergent.		
	\end{proof}
	
We now return to (\ref{eq2.6.12}) and substitute $ \Psi (\widetilde \beta) = \Psi(0) + O( \tau P^{-\theta}).$ In view of (\ref{eq2.6.14}) this yields
	\begin{equation} \label{eq2.6.15}
		\mathcal I_{\pm} = 2 \tau \Omega( s, \theta ;  \tuplelambda) P^{s-\theta} + O \left( P^{s- 2\theta} \right),
	\end{equation}
where
	\begin{equation*}
		\Omega( s, \theta ; \tuplelambda)  =  \left(\frac{1}{\theta}\right)^s | \lambda_1 \cdots \lambda_s|^{-1/\theta} C(s, \theta ; \tuplelambda) \hspace{0.05in} > 0,
	\end{equation*}	 
and $C(s, \theta ; \tuplelambda) = \Psi (0)$ with  $\Psi(0)$ given by (\ref{eq2.6.9}) so that
	\begin{equation*}
		\Psi(0) = \int_{\mathcal U(0) } \left( - \sigma_s( \sigma_1 \beta_1 + \cdots + \sigma_{s-1} \beta_{s-1} ) \right)^{1/ \theta -1} ( \beta_1 \cdots \beta_{s-1})^{1 / \theta -1} \text d \beta_1 \cdots \text d \beta_{s-1}.
	\end{equation*}
Note that $ \Psi(0)$ is positive since not all of the $\sigma_i$ are equal. This can be readily seen as follows. 

Let $ | \lambda_0| = \min_i | \lambda_i| .$  Trivially one has
	\begin{equation} \label{eq2.6.16}	
		\Psi (0) \gg \int_0^{ | \lambda_0|} \cdots  \int_0^{ | \lambda_0| } \left( - \sigma_s( \sigma_1 \beta_1 + \cdots + \sigma_{s-1} \beta_{s-1} ) \right)^{1/ \theta -1} ( \beta_1 \cdots \beta_{s-1})^{1 / \theta -1} \text d \tuplebeta.
	\end{equation}
Since the $ \sigma_i$ are not all of the same sign, by linearity there exists a tuple $ \tuplebeta$ such that 
	\begin{equation*}
		- \sigma_s( \sigma_1 \beta_1 + \cdots + \sigma_{s-1} \beta_{s-1} ) > 0,
	\end{equation*}
with $ 0 < \beta_i \leq | \lambda_0| .$ One can now assume that there exists a large positive number $ D $ that depends on $ \beta_i,$ such that $ \frac{1}{D} \leq \beta_i \leq | \lambda_0|.$ Hence, there exists an open neighbourhood of positive measure over which the integrand on the right hand side of (\ref{eq2.6.16}) is positive. Therefore we deduce that $ \Psi (0) \gg 1.$
	
The asymptotic formula (\ref{eq2.6.15}) together with (\ref{eq2.6.6}) yields
	\begin{equation*}
		\int_{\mathfrak M} f_1(\alpha) \cdots f_s(\alpha) K_\pm(\alpha) \text d \alpha =  2 \tau \Omega( s, \theta ; \tuplelambda) P^{s-\theta} + o \left( P^{s- \theta} \right).
	\end{equation*}
The proof of Theorem \ref{thm2.1.1} is now complete by taking into account Lemma \ref{lemma2.4.3}, Lemma \ref{lemma2.5.1} and the expression (\ref{eq2.2.11}).

%%%%%%%%% END OF MAJOR ARC ANALYSIS %%%%%%%%%%

%%%%%%%%%%%%%% BEGINNING OF INHOMOGENEOUS CASE %%%%%%%%%%%%%%%%%%

\section{The inhomogeneous case} \label{section_inhomog_case}

In this section we prove Theorem \ref{thm2.1.2}. Using the kernel functions defined in (\ref{eq2.2.3}) we have that
	\begin{equation*}
		R_{-} (P) \leq  \mathcal N_{s,\theta}^\tau (P ; \tuplelambda, L) \leq R_{+} (P),
	\end{equation*}
whereas now 
	\begin{equation*}
		R_\pm (P) = \int_{-\infty}^\infty f_1( \alpha) \cdots f_s(\alpha) e(-\alpha L) K_\pm(\alpha) \text d \alpha .
	\end{equation*}
	
To study the above integrals we dissect the real line as in the case of Theorem \ref{thm2.1.1}. By the triangle inequality and appealing to Lemma \ref{lemma2.4.3} and Lemma \ref{lemma2.5.1}, we immediately obtain that
	\begin{equation*}	
		\left| \int_{ \mathfrak m \cup \mathfrak t} f_1 (\alpha) \cdots f_s (\alpha)  e(-\alpha L) K_\pm(\alpha) \text d \alpha \right|  \ll \int_{ \mathfrak m \cup \mathfrak t} \left| f_1 (\alpha) \cdots f_s (\alpha) \right|  \text d \alpha = o \left( P^{s - \theta} \right).
	\end{equation*}
Thus, we are left to deal with the contribution arising when integrating over the major arc around zero. The approach given in section \ref{section_major_arc_asympt_formular} applies here as well with minor adjustments, in order to deal with the factor $ e (- \alpha L).$ We briefly now discuss these differences.
	
The singular integral is now given by
	\begin{equation*} 
		\mathcal I_\pm = \int_{ - \infty}^\infty \upsilon_1 (\alpha) \cdots \upsilon_s (\alpha) e (- \alpha L) K_\pm (\alpha) \text d \alpha,
	\end{equation*}
where the functions $ \upsilon_i (\alpha)$ are defined as in (\ref{eq2.6.1}). One may show as in section \ref{section_major_arc_asympt_formular} that
	\begin{equation*}
		\mathcal I_\pm = \int_{ \mathfrak M}  f_1 (\alpha) \cdots f_s (\alpha) e (- \alpha L) K_\pm (\alpha) \text d \alpha + o \left( P^{s-\theta} \right).
	\end{equation*}
Thus, we aim to give an asymptotic formula for the complete singular integral $ I_\pm$ defined above. 

For $ \alpha \in \mathbb R$ we now define 
	\begin{equation} \label{eq2.7.1}
		\Phi (\alpha) = \upsilon_1 (\alpha) \cdots \upsilon_s (\alpha) e (- \alpha L) = \int_{[0,P]^s} e \left( \alpha \left( \lambda_1 \gamma_1^\theta + \cdots + \lambda_s \gamma_s^\theta -L \right) \right) \text d \tuplegamma.
	\end{equation}
Ignoring for the moment the factor $ e(- \alpha L)$ one can study the function $ \Phi$ as before. This analysis leads now to 
	\begin{equation*}
		\Phi (\alpha) = \left( \frac{P}{\theta} \right)^s | \lambda_1 \cdots \lambda_s |^{-1/\theta} \int_{- \infty}^\infty \Psi (\widetilde \beta) e \left(\alpha \left( P^\theta \sigma_s \widetilde \beta - L \right) \right) \text d \widetilde \beta,
	\end{equation*}
where $ \Psi$ is defined as in (\ref{eq2.6.9}). Applying now Fourier's inversion theorem we obtain that
	\begin{equation} \label{eq2.7.2}
		\mathcal I_\pm =  \left( \frac{P}{\theta} \right)^s | \lambda_1 \cdots \lambda_s |^{-1/\theta} \int_{- \infty}^\infty \Psi (\widetilde \beta) \left(\int_{- \infty}^\infty   e \left(\alpha \left( P^\theta \sigma_s \widetilde \beta - L \right) \right)  K_\pm (\alpha) \text d \alpha \right) \text d \widetilde \beta.
	\end{equation}
	
One may assume that $ \widetilde \beta$ satisfies 
	\begin{equation} \label{eq2.7.3}
		\int_{-\infty}^\infty  e \left( \alpha \left(P^\theta \sigma_s \widetilde \beta - L \right) \right) K_\pm (\alpha) \text d \alpha =
			\begin{cases}
				1,  & \text{if} \hspace{0.05in} | \widetilde \beta - L P^{-\theta}| < \tau P^{-\theta}, \\[10pt]
				0, & \text{otherwise}.				
			\end{cases}
	\end{equation}
Note that for the measure of the set of points $ \widetilde \beta$ which do not satisfy the above relation  one has $ O \left( \tau | L^{-1} | P^{-\theta} \left( \log P \right)^{-1} \right).$ The contribution coming from such $\widetilde \beta$ is $o \left(P^{s-\theta}\right)$ and hence one may ignore this set. Under the assumption that $ | \widetilde \beta - L P^{-\theta}| < \tau P^{-\theta} $ one can show that 
	\begin{equation*} 
		\Psi (\widetilde \beta ) - \Psi (1) \ll  ( \tau + |L^{-1}|) P^{-\theta}.
	\end{equation*}
Indeed, since the factor $ e(-\alpha L)$ in (\ref{eq2.7.1}) does not affect things, one can repeat the argument given in the proof of Lemma \ref{lemma2.6.1} to deduce that
	\begin{equation*}
		\Psi (\widetilde \beta ) - \Psi (0) \ll | \widetilde \beta | \ll | \widetilde \beta - L P^{-\theta}| + |L P^{-\theta}| \ll \left( \tau + |L| \right) P^{-\theta}.
	\end{equation*}
One can now substitute $ \Psi (\widetilde \beta) = \Psi (0) + \left( \tau + |L| \right) P^{-\theta}$ into (\ref{eq2.7.2}). In view of (\ref{eq2.7.3}) this yields
	\begin{equation*}
		\mathcal I_\pm =  \left( \frac{P}{\theta} \right)^s | \lambda_1 \cdots \lambda_s |^{-1/\theta} \int_{ - \tau P^{-\theta} + L P^{-\theta}}^{ \tau P^{-\theta} + LP^{-\theta}}  \left( \Psi (0) + \left( \tau + |L| \right) P^{-\theta} \right) \text d \widetilde \beta.
	\end{equation*}
Therefore, we deduce that
	\begin{equation*}
		\mathcal I_\pm = 2 \tau \Omega(s,\theta ; \tuplelambda)  P^{s -\theta} + O \left( \left( \tau + |L| \right)  P^{s-2 \theta} \right),
	\end{equation*}	
where 
	\begin{equation*}
		\Omega (s, \theta ; \tuplelambda) = \left( \frac{1}{\theta} \right)^s | \lambda_1 \cdots \lambda_s |^{- 1 / \theta} \Psi (0) > 0,
	\end{equation*}	
with $ \Psi (0) $ given by (\ref{eq2.6.9}). The proof of Theorem \ref{thm2.1.2} is now complete.

%%%%%%%%%%%%%%%%%%%%%%%%%%%%%
% END OF INHOMOGENEOUS CASE %	
%%%%%%%%%%%%%%%%%%%%%%%%%%%%%

%%%%%%%%%%%%%%%%%%%%%%%%%%%
% BEGINNING DEFINITE CASE %
%%%%%%%%%%%%%%%%%%%%%%%%%%%

\section{The definite case} \label{section_definite_case}

In this section we prove Theorem \ref{thm2.1.3}. Here we deal with positive definite generalised polynomials. In this section we put
	\begin{equation*}
		\mathcal F (\tuplex) =  \lambda_1 x_1^\theta + \cdots + \lambda_s x_s^\theta - \nu ,
	\end{equation*}
and recall that we  write $ \rho_s (\tau, \nu)$  to denote the number of solutions $ \tuplex \in \mathbb N^s$  possessed by the inequality $ | \mathcal F (\tuplex) | < \tau$ for a fixed real number $ \tau \in (0,1].$ Our approach follows that presented in \cite[Theorem 1.10]{brued_kawada_wooley_8}, where the authors deal with the problem of counting solutions to inequalities for positive definite polynomials.

For any solution $ \tuplex$ counted by $ \rho_s (\tau, \nu)$ one has $ 0 < x_i \leq P \hspace{0.05in} (1 \leq  i\leq s),$ where
	\begin{equation} \label{eq2.8.1}
		P = 2 \left( \lambda_1^{- 1 / \theta} + \cdots + \lambda_s^{ - 1 / \theta} +1 \right) \nu^{ 1/ \theta}.
	\end{equation}
So one can write
	\begin{equation*}
		\rho_s (\tau, \nu) = \sum_{ \substack{ \tuplex \in [1,P]^s \\ | \mathcal F(\tuplex) | < \tau }  }  1.
	\end{equation*}	

Recall the kernel function $ K(\alpha) = \sinc^2(\alpha).$ For any real $ \eta > 0$ we define the function
	\begin{equation*} \label{pos_def_w_def}
		w_\eta (x) = \eta K(\eta x)
	\end{equation*}
that was used in \cite{davenport-heilbr-ineq}. It satisfies
	\begin{equation} \label{eq2.8.2}
		w_\eta (x) \ll \min \{ 1, | x |^{-2}\} \hspace{0.1in} \text{and} \hspace{0.1in} 0 \leq w_\eta(x) \leq \eta.
	\end{equation}
The Fourier transform of this function is given by
	\begin{equation} \label{eq2.8.3}
		\widehat w_\eta (x) = \int_{ - \infty}^{\infty} w_\eta (u) e (- xu) \text d u = \max \left\{ 0, 1 - \frac{ | x |}{\eta} \right\}.
	\end{equation}

Now we define the weighted integral
	\begin{equation*}
		\rho_s^\star (\tau , \nu) = \int_{ - \infty}^{\infty} f_1 (\alpha) \cdots f_s (\alpha) w_\tau (\alpha) \text d \alpha.
	\end{equation*}
In the light of the discussion in \cite[\S 2.1, \S 2.2]{brued_kawada_wooley_8} and appealing to  \cite[Lemma 2.1]{brued_kawada_wooley_8}, whenever $ 0 < \Delta < \frac{ \tau} {2}$ one has 
	\begin{equation} \label{eq2.8.4}
		\rho_s (\tau ,\nu) = \left( 1 + \frac{\tau}{ \Delta} \right) \rho_s^\star \left( \tau + \Delta, \nu \right) - \frac{ \tau}{ \Delta} \rho_s^\star (\tau, \nu) + O \left( \rho_s^\star \left(\Delta, \nu + \tau \right) + \rho_s^\star \left(\Delta, \nu - \tau \right) \right).
	\end{equation}
It is apparent by (\ref{eq2.8.4}) that it is enough to establish an asymptotic formula for the weighted integral $ \rho_s^\star (\tau, \nu).$ To do so, we dissect the real line into three disjoint sets as in section \ref{section_set_up_overview}. Note that now we take $P$ as defined in (\ref{eq2.8.1}). 

For estimating the contribution arising from the sets of minor and trivial arcs one can invoke Lemma \ref{lemma2.4.3} and Lemma \ref{lemma2.5.1}. Together with the fact that by (\ref{eq2.8.2}) one has $  w_\tau (\alpha)  \ll 1 $ for any $\alpha,$ we  deduce that
	\begin{equation} \label{eq2.8.5}
		\int_{ \mathfrak m \cup \mathfrak t} \left| f_1 (\alpha) \cdots f_s (\alpha) e(-\alpha \nu) w_\tau (\alpha) \right| \text d \alpha = o \left( P^{s - \theta} \right).
	\end{equation} 	
So, one is left to deal with the contribution arising when integrating over the major arc. We write
	\begin{equation*}
		I \left( \mathfrak M \right) = \int_{\mathfrak M} f_1 (\alpha) \cdots f_s (\alpha) e(- \alpha \nu) w_\tau (\alpha) \text d \alpha,
	\end{equation*}
and the singular integral is given by
	\begin{equation*}
		\mathcal I_\infty = \int_{- \infty}^{\infty} \upsilon_1 (\alpha) \cdots \upsilon_s (\alpha) e(- \alpha \nu) w_\tau (\alpha ) \text d \alpha,
	\end{equation*}
where the functions $\upsilon_i (\alpha)$ are defined as in (\ref{eq2.6.1}). Below we obtain an asymptotic formula for the integral $ I \left( \mathfrak M \right).$ The argument is analogous to the one given in \cite[Lemma 2.4]{brued_kawada_wooley_8}.

	\begin{lemma} \label{lemma2.8.1}
		Provided that $ s  > 2 \theta$ one has
			\begin{equation*}
				I \left( \mathfrak M \right) =  \frac{ \Gamma \left( 1 + \frac{1}{\theta} \right)^s  }{ \Gamma \left( \frac{s}{\theta}  \right) } \left( \lambda_1 \cdots \lambda_s \right)^{- 1 / \theta} \tau \nu^{ s/ \theta -1} + O \left( \tau \left(  P^{s - \theta - 1 + \delta_0} + P^{s - \theta - \delta_0 (s/ \theta -1) } \right) \right).
			\end{equation*}
		The implicit constant in the error term is independent of $\nu.$
	\end{lemma}

	\begin{proof}
		Using the fact that $ 0 \leq w_\tau (x) \leq \tau$ one has as in (\ref{eq2.6.2}) that
			\begin{equation*}	
					I \left( \mathfrak M \right) - \int_{\mathfrak M}  \upsilon_1 (\alpha) \cdots \upsilon_s (\alpha) e(- \alpha \nu) w_\tau (\alpha) \text d \alpha  \ll \tau P^{s- \theta - 1+ 2 \delta_0}. 
			\end{equation*}
		So as in section \ref{section_major_arc_asympt_formular} we may infer that
			\begin{equation*}
				\begin{split}	
					\int_{ \mathbb R \setminus  \mathfrak M } \upsilon_1 (\alpha) \cdots \upsilon_s (\alpha) e(- \alpha \nu) w_\tau (\alpha) \text d \alpha & \ll \tau \int_{ | \alpha| > P^{- \theta + \delta_0} } | \alpha|^{ - s/ \theta} \text d \alpha \\[10pt]
					& \ll \tau P^{ s - \theta - \delta_0 ( s / \theta -1 )}.
				\end{split}
			\end{equation*}
		Thus, the above two estimates yield	
			\begin{equation} \label{eq2.8.6}
				I \left( \mathfrak M \right) = \mathcal I_\infty + O \left( \tau \left(  P^{s - \theta - 1 + \delta_0} + P^{s - \theta - \delta_0 (s/ \theta -1) } \right) \right) .
			\end{equation}
			
		By (\ref{eq2.6.1}) we may write 
			\begin{equation*}
				\mathcal I_\infty = \int_{ -\infty}^{\infty} \left( \int_{[0,P]^s} e (\alpha (\lambda_1 \gamma_1^\theta+ \cdots + \lambda_s \gamma_s^\theta)) \right) e(- \alpha \nu) w_\tau (\alpha) \text d \alpha.
			\end{equation*}	
		Since the integral is absolutely convergent we can interchange the order of integration in the right hand side of the above formula. Invoking  (\ref{eq2.8.3}) one has
			\begin{equation*}
				\mathcal I_\infty =  \int_{ [0, P]^s } \widehat w_\tau \left( \lambda_1 \gamma_1^\theta + \cdots +  \lambda_s \gamma_s^\theta - \nu \right) \text d \tuplegamma.
			\end{equation*}
		Since $ \lambda_i > 0$ one may use (\ref{eq2.8.3}) to extend the order of integration to $[0,\infty)^s.$ After a change of variables with $ \gamma_i = \beta_i  \lambda_i^{-1 / \theta} \hspace{0.05in} ( 1\leq  i\leq s)$ the above expression takes the shape 
			\begin{equation} \label{eq2.8.7}
				\mathcal I_\infty =  (\lambda_1 \cdots \lambda_s)^{-1 / \theta} \int_{ [0, \infty)^s } \widehat w_\tau \left( \beta_1^\theta + \cdots +  \beta_s^\theta - \nu \right) \text d \tuplebeta.
			\end{equation}
	
		Consider the level sets of the function $ \beta_1^\theta + \cdots + \beta_s^\theta.$ For $ t \in \mathbb R$ the equation  $ t =  \beta_1^\theta + \cdots +  \beta_s^\theta$ defines a surface in $ \mathbb R^s$ of codimension $1.$ We write $ \mathcal S$ to denote the surface obtained by the intersection with the domain $ \{(\beta_1, \ldots, \beta_s) : \beta_i > 0 \hspace{0.05in} (1 \leq i \leq s) \} \subset \mathbb R^s.$ The area of $ \mathcal S $ is equal to 
			\begin{equation*} 
				t^{s / \theta - 1} \frac{ \Gamma \left( 1 + \frac{1}{\theta} \right)^s  }{ \Gamma \left( \frac{s}{\theta}  \right) }. 
			\end{equation*}
		Using the transformation formula we may integrate over $\mathcal S$ and applying Fubini's theorem equation (\ref{eq2.8.7}) takes the shape
			\begin{equation} \label{eq2.8.8}
				\mathcal I_\infty =  (\lambda_1 \cdots \lambda_s)^{-1 / \theta} \frac{ \Gamma \left( 1 + \frac{1}{\theta} \right)^s  }{ \Gamma \left( \frac{s}{\theta}  \right) } \int_0^\infty t^{ s/ \theta -1} \widehat w_\tau ( t - \nu) \text d t.
			\end{equation}
			
		By (\ref{eq2.8.3}) and putting $ t - \nu = u$ one has
			\begin{equation*}
				 \int_0^\infty t^{ s/ \theta -1} \widehat w_\tau ( t - \nu) \text d t = \int_{ - \tau }^{\tau} \left( 1 - \frac{|u |}{\tau} \right) (\nu + u)^{s / \theta-1} \text d u.
			\end{equation*}
		For large enough $ \nu$ one has $ | u/ \nu | < 1 ,$ so the binomial expansion yields
			\begin{equation*}
				(\nu + u)^{ s / \theta -1} = \nu^{s/ \theta - 1} \left( 1 + \frac{u}{\nu} \right)^{s / \theta-1} =  \nu^{s/ \theta - 1} + O \left( u \nu^{ s / \theta -2} \right).
			\end{equation*}
		Hence one has
			\begin{equation*}
				\int_{ - \tau }^{\tau} \left( 1 - \frac{|u |}{\tau} \right) (\nu + u)^{s / \theta-1} \text d u = \tau \nu^{s/ \theta -1} +  O \left( \tau^2 \nu^{ s / \theta -2} \right).
			\end{equation*}
	Returning to (\ref{eq2.8.8}) we deduce that
			\begin{equation*}
				\mathcal I_\infty = (\lambda_1 \cdots \lambda_s)^{-1 / \theta} \frac{ \Gamma \left( 1 + \frac{1}{\theta} \right)^s  }{ \Gamma \left( \frac{s}{\theta}  \right) }   \tau \nu^{s/ \theta -1} +  O \left( \tau^2 \nu^{ s / \theta -2} \right),
			\end{equation*}
		which when combined with (\ref{eq2.8.6}) and observing that $ \tau P^{s - \theta - 1 - \delta_0} \gg \tau^2 \nu^{ s / \theta - 2}$ completes the proof of the lemma.
	\end{proof}

We may now complete the proof of Theorem \ref{thm2.1.3}.
	
	\begin{proof}[Proof of Theorem \ref{thm2.1.3}]
		Putting together (\ref{eq2.8.5}) and the conclusion of Lemma \ref{lemma2.8.1} we deduce that
			\begin{equation*}
				\rho_s^\star (\tau, \nu) = \frac{ \Gamma \left( 1 + \frac{1}{\theta} \right)^s  }{ \Gamma \left( \frac{s}{\theta}  \right) } (\lambda_1 \cdots \lambda_s)^{-1 / \theta} \tau \nu^{s / \theta -1} + o \left( \nu^{ s / \theta -1} \right).
			\end{equation*}
		One can now substitute the above formula into (\ref{eq2.8.4}). This yields
			\begin{equation} \label{eq2.8.9}
				\rho_s (\tau , \nu) = (\lambda_1 \cdots \lambda_s)^{-1 / \theta} \frac{ \Gamma \left( 1 + \frac{1}{\theta} \right)^s  }{ \Gamma \left( \frac{s}{\theta}  \right) }  \left( 2\tau \nu^{s/\theta -1} + \Delta \nu^{s/\theta -1} + W \right)  + o \left( \nu^{s/\theta -1} \right), 
			\end{equation}
		where 
			\begin{equation*}
				W = O \left( \Delta (\nu + \tau)^{s/\theta -1} + \Delta (\nu - \tau)^{s/\theta -1} + o \left( (\nu + \tau)^{s/ \theta -1} \right) + o\left( (\nu - \tau)^{s/ \theta -1} \right) \right).
			\end{equation*}
		Here $ \Delta $ is at our disposal, as long as it satisfies $ 0 < \Delta < \frac{\tau}{2}.$ One may choose $ \Delta = \frac{\tau}{3} P^{-1/100}.$ Then one has
			\begin{equation*} 
				\Delta \nu^{s/\theta -1} +   W =  o \left( \nu^{s/\theta -1} \right).
			\end{equation*}
		Hence the asymptotic formula (\ref{eq2.8.9}) delivers the desired conclusion which completes the proof. 
	\end{proof}
%%%%%%%%%%%%%%%%%%%%%%%%
% END OF DEFINITE CASE %
%%%%%%%%%%%%%%%%%%%%%%%%

%%%%%%%%%%%%%%%%%%%%%%%%%%%%%%%%%%%%%%%%%%%%%%
% BEGINNING DISCRETE L2 RESTRICTION ESTIMATE %
%%%%%%%%%%%%%%%%%%%%%%%%%%%%%%%%%%%%%%%%%%%%%%

\section{A discrete $L^2$-restriction estimate} \label{section_strichartz_est}

This section is devoted to the demonstration of Theorem \ref{thm2.1.5}. Before we present our proof let us motivate the route we take. To make this clearer assume for the moment that $ \theta = d \in \mathbb N.$ Then an application of the Cauchy-Schwarz inequality reveals that
	\begin{equation*}
		\begin{split}
			\int_0^1 \left| \sum_{1 \leq x \leq P} \mathfrak a_x e (\alpha x^d) \right|^{2s} \text d \alpha & = \int_0^1 \left| \sum_{ \ell \in \mathbb Z} \sum_{ \tuplex \in \mathcal B_d (\ell) } \mathfrak a_{x_1} \cdots \mathfrak a_{x_s} e(\alpha \ell) \right|^2 \text d \alpha \\[10pt]
			& \leq \sum_{ \ell \in \mathbb Z} ( \# \left( \mathcal B_d (\ell) \right) ) \sum_{ \tuplex \in \mathcal B_d (\ell) } | \mathfrak a_{x_1} \cdots \mathfrak a_{x_s}|^2,		
		\end{split}
	\end{equation*}
where $ \mathcal B_d (\ell) = \{ 1  \leq \tuplex \leq P : x_1^d + \cdots + x_s^d = \ell \},$ and we write $ \#(\mathcal B_d (\ell))$ to denote its cardinality. By orthogonality one has
	\begin{equation*}
		\#(\mathcal B_d (\ell)) = \int_0^1 \left| \sum_{ 1 \leq x \leq P} e(\alpha x^d)\right|^s e (- \alpha \ell) \text d \alpha.
	\end{equation*}
Hence the problem boils down to bounding the quantity $ \max_{\ell} \#( \mathcal B_d (\ell)).$ Using classical methods together with the circle method, one can show that $\#( \mathcal B_d (\ell)) \ll P^{s-d}$ for sufficiently large $s.$ For example, using the latest method of Wooley \cite{wooley_NEC}  on Vinogradov's mean value theorem, one can take $ s \geq s_0$ with $s_0$ as in (\ref{eq2.1.7}).

When dealing with $ \theta \notin \mathbb N$ one has to modify slightly the argument sketched above. As an analogue of $ \mathcal B_d (\ell)$ we define the set
	\begin{equation*}
		\mathcal B_\theta (\ell) = \left\{ 1 \leq \tuplex \leq P : | x_1^\theta + \cdots + x_s^\theta - \ell | < 1/2 \right\}.
	\end{equation*}
The partition
	\begin{equation*}
		\bigcup_{\ell \in \mathbb Z} \mathcal B_\theta (\ell) = \{ (x_1, \ldots, x_s) :1 \leq x_i \leq P \}
	\end{equation*}	
no longer makes sense for a fractional exponent $ \theta.$ In this situation we instead look at tuples $ \tuplex$ such that $ x_1^\theta + \cdots + x_s^\theta $ is close to an integer value $ \ell.$ This observation makes apparent the link between our aim and the problem of representing integers by a generalized polynomial as described in (\ref{eq2.1.6}). 

Note that with the notation of section \ref{section_definite_case} one has $K(\alpha) = w_1 (\alpha),$ and so by (\ref{eq2.8.3}) one has
	\begin{equation} \label{eq2.9.1}
		\int_{-\infty}^{\infty} e(\alpha \xi) K(\alpha) \text d \alpha = \max \{ 0, 1 - | \xi| \},
	\end{equation}
for all $ \xi \in \mathbb R.$ We may now embark to the proof.

\begin{proof}[Proof of Theorem \ref{thm2.1.5}]
	Recall that we assume  $s \geq 2\left( \lfloor 2\theta  \rfloor+1 \right) \left( \lfloor 2\theta  \rfloor + 2 \right) + 2.$ Expanding one has that
		\begin{equation} \label{eq2.9.2}
			\int_{-\infty}^\infty \left| f_{\mathfrak a}(\alpha) \right|^{2s} K(\alpha) \text d \alpha  = \int_{-\infty}^\infty \left| \sum_{ 1 \leq \tuplex \leq P} \mathfrak a_{x_1} \cdots \mathfrak a_{x_s} e \left( \alpha(x_1^\theta + \cdots + x_s^\theta) \right) \right|^2 K(\alpha) \text d \alpha.			
		\end{equation}
By the definition of the nearest integer function $ \| \cdot \|_{\mathbb R / \mathbb Z} : \mathbb R \to \left[0, 1/2\right]$ we can decompose the summation over $ 1 \leq \tuplex \leq P $ by counting integer solutions of the inhomogeneous inequality 
	\begin{equation*} 
		| x_1^\theta + \cdots + x_s^\theta - \ell_1 | \leq 1/2,
	\end{equation*} 
inside the box $ \left[ 1, P \right]^s,$ where $ \ell_1 $ runs over $ \mathbb Z.$ With this observation and expanding the square, one has that the right hand side of (\ref{eq2.9.2}) is equal to 
	\begin{equation}  \label{eq2.9.3}
		\sum_{ \ell_1, \ell_2 \in \mathbb Z} \hspace{0.05in} \sum_{ \tuplex \in \mathcal B_\theta (\ell_1) } \hspace{0.05in} \sum_{ \tupley \in \mathcal B_\theta (\ell_2) }  \mathfrak a_{x_1} \cdots \mathfrak a_{x_s} \overline{\mathfrak a_{y_1}} \cdots \overline{\mathfrak a_{y_s}} \int_{-\infty}^\infty e \left( \alpha \left( \sigma_{s,\theta}(\tuplex,  \tupley ) \right) \right)  K (\alpha) \text d \alpha,		
	\end{equation}	
where we write $ \sigma_{s,\theta} (\tuplex, \tupley) = x_1^\theta + \cdots + x_s^\theta - y_1^\theta - \cdots - y_s^\theta .$ Let us note that it is at this step where we essentially "double" the number of variables.
	
Invoking (\ref{eq2.9.1}) we see that 
	\begin{equation*}
		0 	< \int_{-\infty}^\infty e \left( \alpha( \sigma_{s,\theta}(\tuplex, \tupley) \right)  K(\alpha) \text d \alpha \leq 1	
	\end{equation*}
if and only if $ | \sigma_{s,\theta} (\tuplex, \tupley) | < 1.$ Indeed, if $ | \sigma_{s,\theta} (\tuplex, \tupley)| \geq 1$ then $ \max \{ 0, 1 - | \sigma_{s,\theta} (\tuplex, \tupley)|  \} =0$ and so by (\ref{eq2.9.1})  we see that the expression in (\ref{eq2.9.3}) is equal to zero. Hence, in this case there is nothing to prove since the estimate claimed in the statement of Theorem \ref{thm2.1.5} trivially holds. Thus, we may assume that the tuples $ \tuplex, \tupley $ satisfy the inequality $  | \sigma_{s,\theta}(\tuplex, \tupley) | < 1.$ Under this assumption one has
	\begin{equation} \label{eq2.9.4}
		\begin{split}	
			\sum_{ \ell_1, \ell_2 \in \mathbb Z} \hspace{0.05in} \sum_{ \tuplex \in \mathcal B_\theta (\ell_1) } \hspace{0.05in} \sum_{ \tupley \in \mathcal B_\theta (\ell_2) }  & \mathfrak a_{x_1} \cdots \mathfrak a_{x_s} \overline{\mathfrak a_{y_1}} \cdots \overline{\mathfrak a_{y_s}}  \int_{-\infty}^\infty  e \left( \alpha \left( \sigma_{s,\theta}(\tuplex, \tupley ) \right) \right)  K (\alpha) \text d \alpha \\[10pt]
			& \ll  \sum_{  \ell_1, \ell_2 \in \mathbb Z } \hspace{0.05in} \mathop{\sum_{\tuplex \in \mathcal B_\theta(\ell_1)} \hspace{0.05in} \sum_{ \tupley \in \mathcal B_\theta(\ell_2)}}_{| \sigma_{s,\theta}(\tuplex, \tupley) | < 1} \mathfrak a_{x_1} \cdots \mathfrak a_{x_s} \overline{ \mathfrak a_{y_1}} \cdots \overline{ \mathfrak  a_{y_s}}.
		\end{split}
	\end{equation}
Let $ (\tuplex,  \tupley) \in \mathcal B_\theta(\ell_1) \times  \mathcal B_\theta(\ell_2)$ and suppose that $ | \sigma_{s,\theta}( \tuplex, \tupley)| <1.$ Then by the triangle inequality one has
	\begin{equation*}
		| \ell_1 - \ell_2 | \leq \left| \ell_1 - (x_1^\theta+ \cdots + x_s^\theta) \right| + \left| \sigma_{s,\theta}( \tuplex, \tupley) \right| +  \left| \ell_2 - (y_1^\theta+ \cdots + y_s^\theta) \right| < 2.
	\end{equation*} 
Therefore, it turns out that 
	\begin{equation} \label{eq2.9.5}
		\begin{split}	
			\sum_{  \ell_1, \ell_2 \in \mathbb Z } \hspace{0.05in} \mathop{\sum_{\tuplex \in \mathcal B_\theta(\ell_1)} \hspace{0.05in} \sum_{ \tupley \in \mathcal B_\theta(\ell_2)}}_{| \sigma_{s,\theta}(\tuplex, \tupley) | < 1} & \mathfrak a_{x_1} \cdots \mathfrak a_{x_s} \overline{ \mathfrak a_{y_1}} \cdots \overline{ \mathfrak  a_{y_s}} \\[10pt]			
			&\leq \sum_{  \ell_1, \ell_2 \in \mathbb Z } \hspace{0.05in} \mathop{\sum_{\tuplex \in \mathcal B_\theta(\ell_1)} \hspace{0.05in} \sum_{ \tupley \in \mathcal B_\theta(\ell_2)}}_{| \ell_1 - \ell_2 | < 2}  \mathfrak a_{x_1} \cdots \mathfrak a_{x_s} \overline{ \mathfrak a_{y_1}} \cdots \overline{ \mathfrak a_{y_s}}.
		\end{split}
	\end{equation}
	
Since $ \ell_1 - \ell_2 \in \{0, \pm 1\}$ we see that if we fix one of the $ \ell_1, \ell_2,$ then the other one has exactly $3$ choices. So by symmetry one has that the expression on the right hand side of (\ref{eq2.9.5}) is bounded above by
	\begin{equation*}
		6 \sum_{ \ell_3 \in \mathbb Z } \left| \sum_{ \tuplex\in \mathcal B_\theta^\prime (\ell_3)}  \mathfrak a_{x_1} \cdots \mathfrak a_{x_s}    \right|^2,
	\end{equation*}
where  for $ \ell_3 \in \mathbb Z$ we put 
	\begin{equation*}
		\mathcal B_\theta^\prime (\ell_3)  = \{ 1 \leq \tuplex  \leq P: | x_1^\theta + \cdots + x_s^\theta - \ell_3 | < 1 \}.
	\end{equation*}
An application of the Cauchy-Schwarz inequality reveals that for $s \geq 2\left( \lfloor 2\theta  \rfloor+1 \right) \left( \lfloor 2\theta  \rfloor + 2 \right) + 2$ one has
	\begin{equation*}
		\begin{split}	
			6 \sum_{ \ell_3 \in \mathbb Z } \left| \sum_{ \tuplex \in \mathcal B_\theta^\prime (\ell_3)}  \mathfrak a_{x_1} \cdots \mathfrak a_{x_s}  \right|^2 & \leq 
			6 \sum_{ \ell_3 \in \mathbb Z} \left( \sum_{ \tuplex \in \mathcal B_\theta^\prime ( \ell_3)} 1 \right) \left(  \sum_{ \tuplex \in \mathcal B_\theta^\prime ( \ell_3)} | \mathfrak a_{x_1} \cdots \mathfrak a_{x_s}|^2 \right)	\\[10pt]
			& \ll \max_{ \ell_3 \in \mathbb Z} \left( \#  \mathcal B_\theta^\prime ( \ell_3) \right) \sum_{ \ell_3 \in \mathbb Z} \sum_{ \tuplex \in \mathcal B_\theta^\prime ( \ell_3)} | \mathfrak a_{x_1} \cdots \mathfrak a_{x_s}|^2 \\[10pt]
			& \ll P^{s-\theta} \left( \sum_{1 \leq  x \leq P} | \mathfrak a_x |^2 \right)^s,							
		\end{split}
	\end{equation*}
where in the last step we used Theorem \ref{thm2.1.3}. Putting together (\ref{eq2.9.4}), (\ref{eq2.9.5}) and invoking (\ref{eq2.9.2}) we are done.
\end{proof}

\newpage 

{\textbf{Acknowledgements.}} This paper is based on work appearing in the author's Ph.D. thesis at the University of Bristol and was supported by a studentship sponsored by a European Research Council Advanced Grant under the European Union's Horizon 2020 research and innovation programme via grant agreement No. 695223. The author would like to thank Prof. Trevor D. Wooley for suggesting this line of research and for the guidance, and Dr. Kevin Hughes for useful discussions and encouragement. The author wishes also to thank Prof. Angel V. Kumchev for detecting an oversight in the application of Lemma \ref{lemma2.3.2} and suggesting the use of a double sum in (\ref{eq2.3.16}). Finally, the author wishes to thank the anonymous referee for reading this manuscript.

\end{document}